\documentclass[a4paper]{amsart}
\usepackage{url}
\usepackage{graphicx}
\usepackage[all]{xy}
\usepackage[mathscr]{eucal}
\usepackage{amsmath,amssymb,amsfonts}

\usepackage{stmaryrd} 

\usepackage{mathrsfs,latexsym,amsthm,enumerate}
\usepackage{amscd}
\newtheorem{theorem}{Theorem}[section]
\newtheorem{lemma}[theorem]{Lemma}
\newtheorem{corollary}[theorem]{Corollary}
\newtheorem{example}[theorem]{Example}

\newtheorem{proposition}[theorem]{Proposition}
\newtheorem{remark}[theorem]{Remark}

\usepackage{stackengine}
\usepackage{xcolor}
\usepackage[all]{xy}

\title{Characterizations of classes of countable Boolean inverse monoids}

\author{Mark V. Lawson}
\address{Mark V. Lawson, Department of Mathematics
and the
Maxwell Institute for Mathematical Sciences,
Heriot-Watt University,
Riccarton,
Edinburgh EH14 4AS,
UNITED KINGDOM}
\email{m.v.lawson@hw.ac.uk}

\author{Philip Scott}
\address{Philip Scott,
Department of Mathematics and Statistics, 
University of Ottawa, 
STEM Complex, 
Ottawa, 
Ontario K1N 6N5, 
CANADA}

\begin{document}
\dedicatory{This paper is dedicated to the memory of Pieter Hofstra, friend and colleague.}

\begin{abstract}
A countably infinite Boolean inverse monoid that can be written as an increasing union
of finite Boolean inverse monoids (suitably embedded) is said to be of finite type.
Borrowing terminology from $C^{\ast}$-algebra theory, we say that such a Boolean inverse monoid is AF (approximately finite)
if the finite Boolean inverse monoids above are isomorphic to finite direct products of finite symmetric inverse monoids,
and we say that it is UHF (uniformly hyperfinite) if the finite Boolean inverse monoids are in fact isomorphic to finite symmetric inverse monoids.
We characterize abstractly the Boolean inverse monoids of finite type and those which are AF and, by using MV-algebras, we also characterize the UHF monoids.
\end{abstract}
\maketitle

\section{Introduction} 

In this paper, we shall be interested in countably infinite Boolean inverse monoids $S$ which can be written in the form $S = \bigcup_{i=1}^{\infty} S_{i}$
where $S_{1} \subseteq S_{2} \subseteq S_{3} \subseteq \ldots$ and where each $S_{i}$ 
is a subalgebra\footnote{This term will be defined later. It is much stronger than merely being an inverse submonoid.} 
of $S$.
The goal of this paper is to characterize $S$ in three different situations:
\begin{enumerate}

\item The $S_{i}$ are finite Boolean inverse monoids. In this case, we say that $S$ is {\em of finite type}.
They are characterized in Theorem~\ref{them:locally-finite}.

\item The $S_{i}$ are {\em matricial semigroups}: that is, finite direct products of finite symmetric inverse monoids.
In this case, we say that $S$ is an {\em AF monoid}. These were the subject of \cite{LS}.
They are characterized in Theorem~\ref{them:locally-matricial}.

\item The $S_{i}$ are finite symmetric inverse monoids.
In this case, we say that $S$ is a {\em UHF monoid}. These are a special case of (2).
They are characterized in Theorem~\ref{thm:UHF}.

\end{enumerate}
Cases (2) and (3) are clearly motivated by the theory of $C^{\ast}$-algebras.

\begin{remark}
{\em The following conventions and definitions are used throughout this paper:
\begin{enumerate}

\item See \cite{Howie1976} for general semigroup theory, which we shall assume;
on a point of notation: if $S$ is a monoid then $\mathsf{U}(S)$ denotes its group of units.

This paper is based on the theory of inverse semigroups (with zero),
and we shall refer to \cite{Lawson1998} when necessary.
We recall a few basic results below.
The only partial order considered in an inverse semigroup is the natural partial order, denoted by $\leq$; meets and joins are always taken with respect to this order.
An element $a$ is an atom if $b \leq a$ implies that $b = a$ or $b = 0$.
The set of idempotents of an inverse semigroup $S$ is denoted by $\mathsf{E}(S)$ and if $X \subseteq S$ then we put $\mathsf{E}(X) = X \cap \mathsf{E}(S)$.
If $a \in S$ then $\mathbf{d}(a) = a^{-1}a$, the {\em domain} of $a$, and $\mathbf{r}(a) = aa^{-1}$, the {\em range} of $a$.
You can check directly that $\mathbf{r}(xy) = \mathbf{r}(x\mathbf{r}(y))$
and 
$\mathbf{d}(xy) = \mathbf{d}(\mathbf{d}(x)y)$.
If $e$ and $f$ are idempotents then $e \stackrel{a}{\longrightarrow} f$ means that $\mathbf{d}(a) = e$ and $\mathbf{r}(a) = f$;
it also means that $e \, \mathscr{D} \, f$.
This enables us to draw pictures of elements of inverse semigroups.
If $\mathbf{d}(a) = \mathbf{r}(b)$, then $ab$ is called the restricted product;
the restricted product $ab$ is usually denoted by $a \cdot b$. 
An inverse semigroup $S$ with respect to the restricted product $(S,\cdot)$ is a groupoid.
Thus groupoids play a role in inverse semigroup theory;
we shall therefore use terminology from groupoid theory where appropriate.
The compatibility relation in $S$ is denoted by $\sim$.
If $e$ is an idempotent then the operation $e \mapsto aea^{-1}$ is called conjugation
and the idempotent $aea^{-1}$ is called the conjugate of $e$.
If $S$ is an inverse monoid, then $\mathsf{U}(S)$ acts on $\mathsf{E}(S)$ by $e \mapsto geg^{-1}$,
where $g$ is a unit.
This is called the {\em natural action}.
If $e$ and $f$ are idempotents, we say they are orthogonal, denoted by $e \perp f$,
if $ef = 0$.
We write $a \perp b$ and say $a$ and $b$ are orthogonal if $\mathbf{d}(a) \perp \mathbf{d}(b)$ and $\mathbf{r}(a) \perp \mathbf{r}(b)$.

\item We shall always insist that $0 \neq 1$ in a Boolean algebra.
The meet operation in a Boolean algebra will be denoted by concatenation; the complement of the element $e$ will be denoted by denoted by $\bar{e}$.
In a Boolean algebra, $ef = 0$ if and only if $e \leq \bar{f}$.

\item $\bigvee \varnothing = 0$ in any poset with a zero.

\item We make a distinction between a partial order and a strict partial order:
thus we write $\subseteq$ for containment and $\subset$ for strict containment.

\item We write $e \preceq f$ if $SeS \subseteq SfS$ where $e$ and $f$ are idempotents.
This is equivalent to $e \, \mathscr{D} \, e' \leq f$ for some idempotent $e'$.
This notation was first used in \cite{LS}.

\item We say that an inverse monoid is {\em directly finite} if $1 \, \mathscr{D} \, e$, where $e$ is an idempotent, implies that $e = 1$.

\item We say that an idempotent $e$ in an inverse semigroup is {\em Dedekind finite} if $f \leq e$ and $e \, \mathscr{D} \, f$ implies that $e = f$.
If every idempotent in an inverse semigroup is Dedekind finite we say that the inverse semigroup is {\em purely Dedekind finite};
the usual term for such inverse semigroups is {\em completely semisimple}. But for Boolean inverse monoids, the term `semisimple' has a different meaning.

\end{enumerate}}
\end{remark}

\noindent
{\bf Acknowledgements} The authors are grateful to the anonymous referee for meticulously reading the various versions of this paper and for the many constructive suggestions
which led to improvements.
Individual contributions by the referee are acknowledged where appropriate.

My coauthor, Phil Scott of the University of Ottawa, sadly passed away in December 2023,
sharp and cheerful until the end.
This version of the paper has therefore not benefited from Phil's eagle-eye.
For that reason, any mistakes or infelicities must be laid at the door of the first named author.

\section{Background results}

The definitions and results of this section are needed to prove
Theorem~\ref{them:locally-finite} and  Theorem~\ref{them:locally-matricial}.
We use the theory of Boolean inverse monoids;
see \cite{Wehrung} for the basics.

We need a few definitions from the theory of posets.
Let $P$ be a poset.
A subset $Q$ of $P$ is called an {\em order-ideal} if $q \in Q$ and $p \leq q$ implies that $p \in Q$.
If $P$ is a poset and $a \in P$, we write $a^{\downarrow}$ for the set $\{b \in P \colon b \leq a\}$.
This is called a {\em principal order-ideal}.
If $P$ and $Q$ are posets then a function $\theta \colon P \rightarrow Q$ is said to be {\em order-preserving}
if $p_{1} \leq p_{2}$ in $P$ implies that $\theta (p_{1}) \leq \theta (p_{2})$ in $Q$.
An {\em order-isomorphism} between two posets is an order-preserving bijection whose inverse is also order-preserving.

Let $A$ be an algebraic structure;
in this paper, $A$ will be either a group or a Boolean algebra or an MV-algebra.
Let $X = \{s_{1}, \ldots, s_{n}\}$ be a finite subset of $A$.
Then $\langle s_{1}, \ldots, s_{n}\rangle$ denotes the substructure of $A$ generated by $X$. 
We say that a structure is {\em locally finite} if every finitely generated substructure is finite.
We follow \cite{CDuM} in our definition of a locally finite MV-algebras.
The following result about Boolean algebras  is proved in \cite[Theorem 2, Chapter 11]{GH}.

\begin{lemma}\label{lem:BA-LF} 
Boolean algebras are locally finite. 
\end{lemma}

We shall need to know the details of the proof of the above.
Let $B$ be a Boolean algebra and let $X = \{p_{1},\ldots, p_{n}\} \subseteq B$.
If $X$ is empty, then $\langle X \rangle = \{0,1\}$.
In what follows, therefore, we shall assume that $X$ is not empty.
A non-zero product $x_{1} \ldots x_{n}$ where each $x_{i}$ is either $p_{i}$ or $\overline{p_{i}}$ is called a {\em minterm}.
The Boolean algebra $\langle X \rangle$ consists of $0$ and $1$, all the minterms (as defined above), and all joins of minterms.

An inverse monoid $S$ is said to be {\em Boolean} if it has all binary compatible joins, multiplication distributes over such joins,
and $\mathsf{E}(S)$ is a Boolean algebra.
If $S$ is a Boolean inverse monoid and $e$ is any non-zero idempotent,
then the local monoid $eSe$ is Boolean with identity $e$.
If $f \in \mathsf{E}(S)$ then the complement of $f$ in $eSe$ is $f' = e\bar{f}$.
A {\em morphism} of Boolean inverse monoids is a homomorphism that preserves
the identity and zero, and maps binary compatible joins to binary compatible joins.

Let $S$ be a Boolean inverse monoid with group of units $\mathsf{U}(S)$ and Boolean algebra of idempotents $\mathsf{E}(S)$.
The natural action of the group $\mathsf{U}(S)$ on the Boolean algebra $\mathsf{E}(S)$,
given by $g \cdot e = geg^{-1}$,
has the following properties.
The proofs of (1) and (2) below are easy:
\begin{enumerate}
\item $g \cdot (e \vee f) = g \cdot e \vee g \cdot f$.
\item $g\cdot (ef) = (g \cdot e)(g \cdot f)$.
\item $g \cdot \bar{e} = \overline{(g \cdot e)}$. 
This is proved from the following (well-known) result:
in a Boolean algebra, if $1 = e \vee f$ and $0 = ef$ then $f = \bar{e}$.
Thus $\bar{e}$ is uniquely determined by the two binary operations.
\end{enumerate}

The symmetric inverse monoid on the {\em non-empty} set $X$ is denoted by $\mathcal{I}(X)$.
This is a Boolean inverse monoid.
If $X$ has $n$ elements (a finite cardinal), then we write $\mathcal{I}_{n}$
and regard the elements, or letters, of $X$ as being elements of the set $\{1,\ldots, n\}$.
The idempotents in $\mathcal{I}_{n}$ are simply the identity functions $1_{A}$ defined on the subsets $A$ of $\{1, \ldots, n\}$.
We have that $1_{A} \,\mathscr{D} \, 1_{B}$ precisely when $A$ and $B$ have the same cardinality.
A semigroup isomorphic to a  finite direct product of finite symmetric inverse monoids is said to be {\em matricial}.

Let $S$ be a Boolean inverse semigroup.
A subset $I \subseteq S$, where $0 \in I$, is called an {\em additive ideal} if $I$ is an ideal of $S$ and
$I$ is closed under binary compatible joins.
If $S$ is a Boolean inverse monoid then both $\{0\}$ and $S$ are additive ideals
and  if these are the only ones we say that $S$ is {\em $0$-simplifying}.
A Boolean inverse semigroup which is both $0$-simplifying and fundamental is called {\em simple}.
The proof of the following can be found in \cite[Theorem 4.18]{Law5}.

\begin{proposition}\label{prop:structure-two} \mbox{}
\begin{enumerate}

\item The finite fundamental Boolean inverse monoids are the matricial inverse monoids.
 
\item The finite simple Boolean inverse monoids are isomorphic to the finite symmetric inverse monoids.

\end{enumerate}
\end{proposition}

An inverse monoid $S$ is said to be {\em factorizable} if every element is below a unit.
The curious terminology stems from the following result, which is well-known and easy to prove.

\begin{lemma}\label{lem:mhg}
Let $S$ be an inverse monoid.
Then $S$ is factorizable if and only if $S = \mathsf{U}(S)\mathsf{E}(S)$.
\end{lemma}
\begin{proof}
Suppose first that $S$ is factorizable.
Then, by definition, if $s \in S$ then $s \leq g$ where $g$ is a unit.
Thus $s = g \mathbf{d}(s) \in  \mathsf{U}(S)\mathsf{E}(S)$.
Conversely, if $S = \mathsf{U}(S)\mathsf{E}(S)$,
then if $s \in S$ we may write $s = ge$ where $g$ is a unit and $e$ is an idempotent.
Thus $s \leq g$.
\end{proof}

Factorizable inverse monoids are therefore determined by groups and meet semilattices.
We shall use these ingredients to build inverse submonoids of inverse monoids.

\begin{lemma}\label{lem:kamala} Let $S$ be an inverse monoid.
Let $G \leq \mathsf{U}(S)$ be a subgroup and let $E \subseteq \mathsf{E}(S)$ be a subsemilattice that contains the identity.
Then $T = GE$ is an inverse submonoid of $S$ (necessarily factorizable) if and only if $E$ is closed under the natural action by $G$.
Furthermore, if $T$ is a submonoid then $\mathsf{U}(T) = G$ and $\mathsf{E}(T) = E$.
\end{lemma}
\begin{proof} Suppose first that $E$ is closed under the natural action by $G$.
Let $ge$ and $hf$ be elements of $T$.
Then $(ge)(hf) = gh(h^{-1}eh)f$.
But $E$ is closed under the natural action by $G$ and so $(ge)(hf) \in T$.
Thus $T$ is closed under multiplication.
Let $ge \in T$.
Then $(ge)^{-1} = g^{-1}(geg^{-1})$.
But $E$ is closed under the natural action by $G$ and so $(ge)^{-1} \in T$.
We have therefore proved that $T$ is an inverse submonoid of $S$.
We now prove the converse. 
Suppose that $T = GE$ is an inverse submonoid.
Let $g \in G$ and $e \in E$.
Put $s = ge \in T$.
Then $ss^{-1} \in T$, since $T$ is an inverse submonoid of $S$, from which it follows that $geg^{-1} \in E$.
We now prove the last claims.
We thank the referee for the proofs.
Let $s \in T$.
Then $s = ge$, where $g \in G$ and $e \in E$.
Suppose that $s$ is a unit in $T$.
We have that $e = 1$.
We deduce that $s \in G$.
Suppose that $s$ is an idempotent.
Then $s^{-1} = s$.
Thus $s = ss = geg^{-1}$.
But, by assumption, this is an element of $E$.
\end{proof}

Part (1) of the following was proved in \cite{DA}.
The remaining proofs are easy.

\begin{lemma}\label{lem:metal}\mbox{}
\begin{enumerate}
\item The symmetric inverse monoid $\mathcal{I}(X)$ is factorizable if and only if $X$ is finite.
\item Finite direct products of factorizable inverse monoids are factorizable. 
\item Matricial monoids are factorizable.
\end{enumerate}
\end{lemma}

The following was proved in \cite{DA}.

\begin{lemma}\label{lem:fac-is-df} 
Every factorizable inverse monoid is directly finite.
\end{lemma}

We say that a Boolean inverse monoid $S$ has {\em $\mathscr{D}$-complementation} if $e \, \mathscr{D} \, f$ implies that $\bar{e} \, \mathscr{D} \, \bar{f}$
for any idempotents $e,f$.
The proof of the following can be found in \cite[Proposition 2.7]{LS}.

\begin{proposition}\label{prop:factorizable} Let $S$ be a Boolean inverse monoid.
Then $S$ is factorizable if and only if $S$ has $\mathscr{D}$-complementation.
\end{proposition}

The following is important, but the proof given in  \cite[part (3) of Lemma 2.6]{LS} is wrong.

\begin{lemma}\label{lem:dedef}
Every factorizable Boolean inverse monoid is purely Dedekind finite.
In particular, $\mathscr{J} = \mathscr{D}$.
\end{lemma}
\begin{proof} Suppose that $e \, \mathscr{D} \, f \leq e$.
Let $e \stackrel{a}{\rightarrow} f$.
Observe that $a \in eSe$.
Thus $g = a \vee \bar{e}$ is a well-defined join.
But $\mathbf{d}(g) = e \vee \bar{e} = 1$.
However, by Lemma~\ref{lem:fac-is-df}, factorizable inverse monoids are directly finite.
Thus $\mathbf{r}(g) = 1$.
We have therefore proved that $f \vee \bar{e} = 1$.
If we multiply both sides of this equation by $e$,
we have that $ef = e$ and so $e \leq f$.
We deduce that $e = f$.
The final claim follows from \cite[Corollary 3.2.19]{Lawson1998}
\end{proof}

\begin{lemma}\label{lem:factor-mu} 
Let $S$ be a Boolean inverse monoid.
Then $S$ is factorizable if and only if $S/\mu$ is factorizable.
\end{lemma}
\begin{proof} Only one direction needs proving.
Let $e \, \mathscr{D} \, f$ in $S$ where $e$ and $f$ are idempotents.
Then $\mu (e) \, \mathscr{D} \, \mu (f)$.
The inverse monoid $S/\mu$ is Boolean by \cite[Proposition 3.4.5]{Wehrung}
and, by assumption, $S/\mu$ is factorizable.
Thus $\overline{\mu (e)}\, \mathscr{D} \, \overline{\mu (f)}$ by Proposition~\ref{prop:factorizable}.
But $\mu$ induces an isomorphism of the Boolean algebras of idempotents.
Thus $\mu (\bar{e})\, \mathscr{D} \, \mu (\bar{f})$ in $S/\mu$.
This implies that $\bar{e} \, \mathscr{D} \, \bar{f}$ in $S$.
Thus $S$ is factorizable by Proposition~\ref{prop:factorizable}.
\end{proof}

The following is our first important result.

\begin{proposition}\label{prop:finite-factor} 
All finite Boolean inverse monoids are factorizable.
\end{proposition}
\begin{proof} 
Let $S$ be a finite Boolean inverse monoid.
Then $S/\mu$ is a finite fundamental Boolean inverse monoid.
Thus by Proposition~\ref{prop:structure-two}, it is matricial
and so it is factorizable by  Lemma~\ref{lem:metal}. 
The result now follows by Lemma~\ref{lem:factor-mu}.
\end{proof}

Let $B$ be a Boolean algebra.
We say that $B'$ is a {\em subalgebra} of $B$ if $B' \subseteq B$  contains the 0 and 1 and is closed
under joins, meets and complements.
Now, let $S$ and $T$ be Boolean inverse monoids where $T$ is an inverse submonoid of $S$.
We say that $T$ is a {\em subalgebra} of $S$ if two conditions are satisfied:
\begin{enumerate}
\item If $a,b \in T$ and $a \sim b$ then $a \vee b \in T$, where $a \vee b$ means the join calculated in $S$.
\item $\mathsf{E}(T)$ is a subalgebra of $\mathsf{E}(S)$;
because we are assuming that $T$ is closed under joins and products, in the presence of (1) this is equivalent to requiring that if  
$e \in \mathsf{E}(T)$ then $\bar{e} \in \mathsf{E}(T)$.
\end{enumerate}
Wehrung \cite[Definition 3.1.7]{Wehrung} uses the terminology `additive inverse subsemigroup'.
The justification for our use of the term `subalgebra' is based on \cite[Corollary 3.2.7]{Wehrung}.

Let $S$ be a Boolean inverse monoid and $X \subseteq S$.
Then $X^{\vee}$ is the set of all joins of finite compatible subsets of $S$.
Clearly, $X \subseteq X^{\vee}$.

\begin{lemma}\label{lem:gillian} Let $S$ be a Boolean inverse monoid.
Let $T$ be an inverse submonoid of $S$ having the same zero and such that $\mathsf{E}(T)$ is a Boolean subalgebra of $\mathsf{E}(S)$.
Then $T^{\vee}$ is a subalgebra of $S$.
\end{lemma}
\begin{proof} It is enough to prove that $T^{\vee}$ is a Boolean inverse submonoid of $S$.
It is clear that $\mathsf{E}(T^{\vee}) = \mathsf{E}(T)$, 
by construction $T^{\vee}$ has all joins, 
and the fact that multiplication
distibutes over binary joins is inherited from $S$.
Thus $T^{\vee}$ is a Boolean inverse submonoid of $S$ and a subalgebra.
\end{proof}

The following lemma tells us how to construct finite Boolean inverse subalgebras of Boolean inverse monoids in certain
circumstances.

\begin{lemma}\label{lem:nb} Let $S$ be a Boolean inverse monoid.
\begin{enumerate}

\item Let $G \leq \mathsf{U}(S)$ be a subgroup and let $E \subseteq \mathsf{E}(S)$ be a finite Boolean subalgebra.
Then there is a finite Boolean subalgebra $B$ of $\mathsf{E}(S)$ such that $E \subseteq B$ and $B$ is invariant under the natural action of $G$.

\item Let $G \leq \mathsf{U}(S)$ be a finite subgroup and let $E \subseteq \mathsf{E}(S)$ be a finite Boolean subalgebra.
Then there is a finite Boolean inverse subalgebra $T$ of $S$ 
such that $GE \subseteq T$.

\end{enumerate}
\end{lemma}
\begin{proof} (1) Let $B$ be the Boolean subalgebra of $\mathsf{E}(S)$ generated by the finite set 
$$\{geg^{-1} \colon g \in G, e \in E\}.$$
This is finite by Lemma~\ref{lem:BA-LF}.
By construction, $B$ is closed under the natural action by $G$.

(2) Construct $B$ from $E$ using part (1) above.
Then $F = GB$ is an inverse submonoid of $S$, by Lemma~\ref{lem:kamala},
with group of units $G$ and set of idempotents  $B$.
By Lemma~\ref{lem:gillian}, $F^{\vee}$ is a subalgebra of $S$.
\end{proof}

Let $S$ be an inverse semigroup.
We say that $S$ is a {\em meet-semigroup} if each pair of elements has a meet.
A function $\phi \colon S \rightarrow \mathsf{E}(S)$ is called a {\em fixed-point operator} if $\phi (s)$
is the largest idempotent below $s$.
Leech \cite{Leech} proved  
that an inverse semigroup has all binary meets if and only if it is equipped with a fixed-point operator;
in fact, if a fixed-point operator exists, then $a \wedge b = \phi (ab^{-1})b$.
In the case of inverse meet-monoids,
$\phi (s) = s \wedge 1$.

A non-zero element $a$ of an inverse semigroup is said to be an {\em infinitesimal} if $a^{2} = 0$ \cite{Lawson2016}.
Observe that the non-zero element $a$ is an infinitesimal if and only if $\mathbf{d}(a) \perp \mathbf{r}(a)$.
A Boolean inverse semigroup is said to be {\em basic} if each non-zero element can 
be written as a join of an idempotent and a finite number of infinitesimals.

A Boolean inverse monoid $S$ is said to be {\em piecewise factorizable} if every element $s \in S$ can be written
$s = \bigvee_{i=1}^{n} g_{i}e_{i}$ where the $g_{i}$ are units and the $e_{i}$ are idempotents.

For part (1) of the following, see \cite[Remark 4.29]{Lawson2016}.
The proofs of the other parts are straightforward.

\begin{lemma}\label{lem:basic} \mbox{}
\begin{enumerate}
\item Finite symmetric inverse monoids are basic.
\item A finite direct product of basic Boolean inverse monoids is a basic Boolean inverse semigroup.
\item Matricial monoids are basic.
\end{enumerate}
\end{lemma}

The following is proved as \cite[Lemma 4.30]{Lawson2016}.

\begin{proposition}\label{prop:basic-properties} Let $S$ be a basic Boolean inverse monoid.
Then $S$ is:
\begin{enumerate}
\item A meet-monoid; this means that basic Boolean inverse monoids are equipped with a fixed-point operator $\phi$.
\item Piecewise factorizable.
\item Fundamental.
\end{enumerate}
\end{proposition}

For finite Boolean inverse monoids, however, being basic and being fundamental are the same thing.

\begin{proposition}\label{prop:covid} Let $S$ be a finite Boolean inverse monoid.
Then $S$ is fundamental if and only if $S$ is basic.
\end{proposition}
\begin{proof} By Proposition~\ref{prop:basic-properties}, every basic Boolean inverse semigroup
is fundamental.
To prove the converse, 
let $S$ be a finite fundamental Boolean inverse monoid.
Then $S$ is isomorphic to a matrical monoid by part (1) of Proposition~\ref{prop:structure-two}.
We now use Lemma~\ref{lem:basic} to deduce that $S$ is basic.
\end{proof}

The idempotent which appears in the definition of a basic Boolean inverse monoid
has a natural  interpretation.

\begin{lemma} Let $S$ be a basic Boolean inverse monoid.
If $a \in S$ and 
$$a = a_{1} \vee \ldots \vee a_{m} \vee e,$$
where the $a_{1}, \ldots , a_{m}$ are infinitesimals and $e$ is an idempotent, then $e = \phi (a)$.
\end{lemma}
\begin{proof} Clearly,  $e \leq a$ and so $e \leq \phi (a)$.
Let $f \leq a$ be any idempotent.
Then $f = af$ and so 
$f = a_{1}f \vee \ldots \vee a_{m}f \vee ef$.
Each $a_{i}f$ is an idempotent (since it is less than an idempotent) which is either zero or an infinitesimal.
We show that they must be zero.
Let $b \leq f$ where $b$ is either zero or an infinitesimal.
Then $b = bf = fb$.
But $b$ is an idempotent.
Thus $b = fbbf = 0$.
It follows that $f = ef$ and so $f \leq e$.
This proves that $e = \phi (a)$.
\end{proof}

\section{Characterization theorems}

In this section, we characterize the Boolean inverse monoids of finite type and those which are AF.
The characterization of UHF monoids is more complex and left until Section~4.

\begin{theorem}[Boolean inverse monoids of finite type]\label{them:locally-finite} Let $S$ be a countably infinite Boolean inverse monoid.
Then $S$ is of finite type if and only if $S$ is factorizable and its group of units is locally finite.
\end{theorem}
\begin{proof} We prove the easy direction first.
Suppose that $S$ is of finite type.
Then $S = \bigcup_{i=1}^{\infty} S_{i}$ where $S_{1} \subseteq S_{2} \subseteq S_{3} \subseteq \ldots$, 
and each $S_{i}$ is a finite subalgebra of $S$.
Let $a \in S$.
Then $a \in S_{i}$, for some $i$.
All finite Boolean inverse monoids $S_{i}$ are factorizable by Proposition~\ref{prop:finite-factor}.
Thus $a \leq g$ where $g$ is a unit of $S_{i}$.
However, $g$ is also a unit of $S$ since $S_{i}$ and $S$ have the same identity.
Thus $S$ is also factorizable. 
Now, let $X$ be any finite subset of the group of units of $S$.
Then $X \subseteq S_{i}$, for some $i$, where $S_{i}$ is a finite subalgebra of $S$.
It follows that $X$ are also units of $S_{i}$.
Hence, the subgroup generated by $X$ is finite.
We have therefore proved that the group of units of $S$ is locally finite.

We now prove the converse.
Let $S$ be a countably infinite Boolean inverse monoid
which is factorizable and whose group of units is locally finite.
We prove that $S$ is of finite type.
Since $S$ is countably infinite, both the group $\mathsf{U}(S)$ and the Boolean algebra $\mathsf{E}(S)$ are countable.
We may therefore list the elements of both:
let $\mathsf{U}(S) = \{g_{1} = 1, g_{2}, \ldots \}$ and $\mathsf{E}(S) = \{e_{1} = 1, e_{2} \ldots \}$.
Define $G_{i} = \langle g_{1}, \ldots, g_{i} \rangle$ and $E_{i} = \langle e_{1}, \ldots, e_{i}\rangle$.
Since $\mathsf{U}(S)$ is locally finite by assumption, the groups $G_{i}$ are all finite. 
The Boolean algebras $E_{i}$ are finite by Lemma~\ref{lem:BA-LF} since they are finitely generated.
By construction $G_{1} \subseteq G_{2} \subseteq \ldots$ and $E_{1} \subseteq E_{2} \subseteq \ldots$.
By Lemma~\ref{lem:nb},  we may enlarge each Boolean algebra $E_{i}$ to a finite Boolean algebra $E_{i}'$
where $E_{i}'$ is closed under the natural action by $G_{i}$ and such that $E_{1}' \subseteq E_{2}' \subseteq \ldots$:
to do this, first enlarge $E_{1}$ to $E_{1}'$ and then enlarge $\langle E_{1}' \cup E_{2} \rangle$ to get $E_{2}'$
and then enlarge $\langle E_{2}' \cup E_{3} \rangle$ to get $E_{3}'$ and so on.
We relabel $E_{i}'$ as $E_{i}$. 
We may therefore assume that each $E_{i}$ is closed under the natural action of $G_{i}$.
By Lemma~\ref{lem:kamala},
there is an increasing sequence $F_{i} = G_{i}E_{i}$
of finite inverse submonoids where $F_{1} \subseteq F_{2} \subseteq F_{3} \subseteq \ldots$,
each $F_{i}$ is factorizable, and $\mathsf{E}(F_{i})$ is a subalgebra of $S$.
By Lemma~\ref{lem:nb},
there is an increasing sequence $F_{1}^{\vee} \subseteq F_{2}^{\vee} \subseteq \ldots$
of finite Boolean inverse subalgebras of $S$.
Let $a \in S$ be arbitrary.
Then $a = ge$ for some unit $g$ and idempotent $e$ by factorizability.
It follows that $a \in F_{i} \subseteq F_{i}^{\vee}$ for some $i$.
We have therefore proved that $S$ is of finite type.
\end{proof}

The referee asked whether Boolean inverse monoids of finite type might be meet-monoids.
We think not, although we do not have a counterexample.
The following example does show that you have to be careful with meets.

\begin{example}\label{ex:counterexample}
{\em  Consider the finite Boolean inverse meet-monoid on 3 letters $\mathcal{I}_{3}$.
Denote the identity of this monoid by $1$ and its zero by $0$.
Put $g = (2,3)$ a transposition, and $G = \{1,g\}$.
Then $G$ is a subgroup of the group of units of $\mathcal{I}_{3}$.
If we adjoin the zero, we get that $G^{0}$ is a group with zero adjoined contained in $\mathcal{I}_{3}$.
The meet $g \wedge 1$ calculated  in $G^{0}$ is $0$, 
but $g \wedge 1$ calculated in $\mathcal{I}_{3}$ is the idempotent $e = 1_{\{1\}}$.
}
\end{example}

We now turn to the characterization of AF inverse monoids, which requires the following result first.
We have followed the referee's suggestion and only assumed piecewise factorizability rather than that of being basic.

\begin{lemma}\label{lem:basic-factorizable} Let $S$ be a piecewise factorizable Boolean inverse monoid whose group of units is locally finite.
Then $S$ is factorizable.
\end{lemma}
\begin{proof} Let $a \in S$.
By definition,  
we may write
$a = \bigvee_{i=1}^{n} g_{i}e_{i}$ where the $g_{i}$ are units and the $e_{i}$ are idempotents. 
Put $G = \langle g_{1}, \ldots, g_{n} \rangle$, which is finite since the group of units of $S$ is locally finite. 
Put $\mathscr{E} = \{e_{1}, \ldots, e_{n}\}$.
By Lemma~\ref{lem:nb}, there is a finite Boolean subalgebra $B$ of $\mathsf{E}(S)$ containing $\mathscr{E}$
invariant under the natural action by $G$.
By Lemma~\ref{lem:kamala},
$GB$ is a  finite inverse submonoid of $S$, which is factorizable by construction
and has as its semilattice of idempotents the Boolean subalgebra $B$.
Put $T = (GB)^{\vee}$.
Then by Lemma~\ref{lem:nb}, $T$ is a finite Boolean inverse submonoid containing $a$.
Thus, in particular, $T$ is a factorizable Boolean inverse monoid by Proposition~\ref{prop:finite-factor}.
It follows that $a$ is beneath a unit of $T$ and so beneath a unit of $S$.
The result now follows.
\end{proof}

The following result solves the problem of embedding finite subalgebras into finite fundamental subalgebras.

\begin{lemma}\label{lem:needed-to-prove-result} Let $S$ be a basic Boolean inverse monoid with a locally finite group of units.
Let $T$ be a finite subalgebra of $S$.
Then there is a finite fundamental subalgebra $T'$ of $S$ such that $T \subseteq T'$.
\end{lemma}
\begin{proof} Since $T$ is a finite Boolean inverse monoid, we know that it is factorizable by Proposition~\ref{prop:finite-factor}.
By Lemma~\ref{lem:mhg}, we can therefore write $T = \mathsf{U}(T) \mathsf{E}(T)$
where $\mathsf{U}(T)$ is a finite subgroup of $\mathsf{U}(S)$ and $\mathsf{E}(T)$
is a finite Boolean subalgebra of the Boolean algebra $\mathsf{E}(S)$.
By Lemma~\ref{lem:kamala}, $\mathsf{E}(T)$ is closed under the natural action by $\mathsf{U}(T)$. 

By assumption, $S$ is basic.
Thus for each 
$g \in \mathsf{U}(T)$, we may write $g = a_{1} \vee \ldots \vee a_{m} \vee e$
where the $a_{1},\ldots, a_{m}$ are infinitesimals in $S$ and $e$ is an idempotent in $S$.
Since $g$ is a unit, $\mathbf{d}(g) = 1$.
Thus  
$1 = \mathbf{d}(a_{1}) \vee \ldots \vee \mathbf{d}(a_{m}) \vee e$
using standard properties of Boolean inverse monoids.
Define $\iota (g) = \{\mathbf{d}(a_{1}), \ldots, \mathbf{d}(a_{m}), e\}$.
Put $\mathsf{E}(T)'$ equal to the set $\mathsf{E}(T)$ together with the sets $\iota (g)$ for each of the $g \in \mathsf{U}(T)$. 
Then adjoin all the conjugates of these idempotents by elements of $\mathsf{U}(T)$.
Now, define $B$ to be the finite Boolean algebra generated by these idempotents.
Thus 
$$\mathsf{E}(T) \cup \left(\bigcup_{g \in \mathsf{U}(T)} \iota (g) \right) \subseteq B.$$
By Lemma~\ref{lem:kamala},
$T_{1}  = \mathsf{U}(T)B$ is a factorizable inverse submonoid of $S$ with $B$ as its Boolean algebra of idempotents
and $\mathsf{U}(T)$ as its group of units.
Put $T' = T_{1}^{\vee}$.
By Lemma~\ref{lem:nb},
this is a finite subalgebra of $S$ that contains $T_{1}$.
It remains to show that $T'$ is fundamental.
We shall do this by showing that $T'$ is basic and then using Proposition~\ref{prop:covid}.
Each element of $a \in T'$ is of the form $a = \bigvee_{i=1}^{m} g_{i}e_{i}$
where $g_{i} \in \mathsf{U}(T)$ and $e_{i} \in B$.
Thus if we can show that each $g_{i}e_{i}$ is a join of infinitesimals and an idempotent
then the element $a$ will also be a join of infinitesimals and an idempotent, and we shall be done.
In what follows, therefore, let $g \in \mathsf{U}(T)$ and $f \in B$.
We prove that $gf$ can be written as a sum of infinitesimals and an idempotent.
By assumption $\iota (g) \subseteq B$.
It follows that in $T'$, we can write $g = a_{1} \vee  \ldots \vee a_{m} \vee e$
where $a_{1}, \ldots, a_{m} \in \mathsf{U}(S)B$ are infinitesimals and $e \in B$.
If $f \in B$ then $gf = a_{1}f \vee \ldots \vee a_{m}f \vee ef$.
Any element less than an infinitesimal is either zero or an infinitesimal.
If we weed out the products which are zero, we get a representation of $gf$ as a sum of infinitesimals and an idempotent.
We have therefore proved that $T'$ is basic. 
\end{proof}

We now come to our second main theorem.

\begin{theorem}[AF monoids]\label{them:locally-matricial} 
Let $S$ be a countably infinite Boolean inverse monoid.
Then $S$ is an AF monoid if and only if $S$ is basic and its group of units is locally finite.
\end{theorem}
\begin{proof} We prove the easy direction first. 
Let $S$ be an AF inverse monoid.
Then $S = \bigcup_{i=1}^{\infty} S_{i}$ where each $S_{i}$ is a matricial subalgebra of $S$.
By Proposition~\ref{prop:covid}, each $S_{i}$ is basic
and so, since $S_{i}$ is a subalgebra of $S$, it follows that $S$ is basic.
The fact that the group of units of $S$ is locally finite follows by the same argument as that used in the proof of Theorem~\ref{them:locally-finite}.

We now prove the converse.
By assumption, our semigroup is basic and so it is piecewise-factorizable
by Proposition~\ref{prop:basic-properties}.
By Lemma~\ref{lem:basic-factorizable}, $S$ is factorizable.
Thus by Theorem~\ref{them:locally-finite}, we can write $S = \bigcup_{i=1}^{\infty} S_{i}$ where the $S_{i}$ are finite subalgebras of $S$. 
By Lemma~\ref{lem:needed-to-prove-result},
every finite subalgebra of $S$ is contained in a finite fundamental subalgebra of $S$.
Thus, we may find a finite fundamental subalgebra $T_{1}$ of $S$ such that $S_{1} \subseteq T_{1}$.
By Proposition~\ref{prop:structure-two}, $T_{1}$ is matricial.
Because $T_{1}$ is finite, we may find an $S_{i}$ such that $T_{1} \subseteq S_{i}$.
By a similar argument to that above, we may find a finite matricial subalgebra $T_{2}$ of $S$
such that $S_{i} \subseteq T_{2}$.
And so on.
We have therefore shown that $S = \bigcup_{i=1}^{\infty} T_{i}$, where each $T_{i}$ is a finite matricial subalgebra of $S$.
\end{proof}

The referee asked if we could replace basic inverse monoids by fundamental inverse monoids in our theorem above.
We have not been able to do so;
nevertheless, the next result explains why basic inverse monoids are natural.
See \cite[Proposition 4.31]{Lawson2016}.\footnote{The result stated in that paper misses out the extra needed condition that $S$ is a meet-monoid.}
This result uses non-commutative Stone duality which is described in \cite{Lawson2022};
the Stone groupoid of a Boolean inverse monoid $S$ is an \'etale topological groupoid
contructed from $S$ using ultrafilters
whose space of identities is homeomorphic to the Stone space of $\mathsf{E}(S)$.
A groupoid is principal if the local groups are trivial.

\begin{proposition}\label{prop:principal} 
Let $S$ be a Boolean inverse meet-monoid.
Then $S$ is basic if and only if the Stone groupoid of $S$ is principal.
\end{proposition}

\section{Characterization of UHF inverse monoids}

In this section, we shall abstractly characterize UHF monoids.
We use Theorem~\ref{them:locally-matricial} from the previous section.
The key to our result is the theory of MV-algebras, so it is there we shall start.

\subsection{MV-algebras}

We refer the reader to \cite{CDM} for background information on MV-algebras.
An {\em MV-algebra} $(A,+,\neg,0)$ is a commutative monoid $(A,+,0)$
together with a unary operation $\neg$ satisfying the following three axioms; in what follows,
we define $1 = \neg 0$:
\begin{enumerate} 
\item $\neg \neg x = x$.
\item $x + 1 = 1$.
\item $\neg (\neg x + y) + y = \neg (\neg y + x) + x$.
\end{enumerate}
If $x$ is an element of an MV-algebra and $n$ is a natural number, define
$$nx = \overbrace{x + \ldots + x}^{n \, \text{\small times}}.$$
By \cite[Lemma 1.1.2]{CDM}, 
we may define a partial order on an MV-algebra by $x \leq y$ if $y = z + x$ for some $z$.
It can be shown that in an MV-algebra, $x \leq y$ if and only if $\neg x + y = 1$ \cite{CDM}.
The unit interval $[0,1]$ becomes an MV-algebra with respect to the following definitions:
$$x + y = \mbox{min}(x+y,1) \, \text{and} \, \neg x = 1-x;$$
observe that above we have used `$+$' in two different ways.
The partial order is the usual partial order.
An {\em ideal} $I$ of an MV-algebra $A$ is a subset that contains $0$,
is an order-ideal and is closed under $+$.
As usual, we shall say that an MV-algebra is {\em simple} if it has no non-trivial ideals.
For $n \geq 2$, define $L_{n}$ to be the MV-algebra with cardinality $n$ that contains
the elements $\frac{r}{n-1}$ where $0 \leq r \leq n-1$ and which is an MV-subalgebra of $[0,1]$.
These are precisely the finite simple MV-algebras by \cite[Corollary 3.5.4]{CDM}. 
We shall use the following property of finite simple MV-algebras later.
The proof is immediate from the structure of such algebras.

\begin{lemma}\label{lem:trump}
Every element of $L_{n}$ is a natural number multiple of the smallest non-zero element of $L_{n}$.
\end{lemma}

More generally, we have the following by \cite[Theorem 3.5.1]{CDM}:

\begin{proposition}\label{prop:simple-MV}
The simple MV-algebras are precisely those which can be embedded in the MV-algebra $[0,1]$.
\end{proposition}

By a {\em rational MV-algebra} 
we mean an MV-algebra isomorphic to a subalgebra of $\mathbb{Q} \cap [0,1]$.
The following result is key to our work.
It is a simple consequence of \cite[Theorem 5.1]{CDuM}. 

\begin{proposition}\label{prop:locally-finite-MV} 
The rational MV-algebras are precisely the simple
locally finite MV-algebras.
\end{proposition}

The following is derived from \cite[Proposition 3.5.3]{CDM}.

\begin{lemma}\label{lem:zelensky} The only finite MV-algebras that can be embedded in the MV-algebra $[0,1]$
are those of the form $L_{n}$ for some $n \geq 2$.
\end{lemma}

\subsection{Foulis monoids}

We shall now define a class of Boolean inverse monoids which give rise to (all) MV-algebras.

We say that a Boolean inverse monoid $S$ satisfies the {\em lattice condition} if  
$S/\mathscr{J}$ is a lattice under subset-inclusion.\footnote{The reader should note that this is an ordering on the principal ideals of the inverse semigroup
and is therefore not the same as the natural partial order; for example, 
in the symmetric inverse monoid $\mathcal{I}_{2}$ the idempotents $e = 1_{\{1\}}$ and $f = 1_{\{2\}}$ are such that $e \wedge f = 0$
but they generate the same principal ideal since $e\,\mathscr{D}\,f$.}
The proof of the following is immediate from the fact that every ideal of a finite symmetric inverse monoid is principal
and is generated by any element of maximum rank.

\begin{lemma}\label{lem:jade} The set of principal ideals of finite symmetric inverse monoids
is linearly ordered.
\end{lemma}

A {\em Foulis monoid} is defined to be a factorizable Boolean inverse monoid 
satisfying the lattice condition.\footnote{This is different from the way the term was used in \cite{LS}
where the `lattice condition' was omitted from the definition.}
Foulis monoids are directly finite by Lemma~\ref{lem:fac-is-df};
they are also purely Dedekind finite and so $\mathscr{D} = \mathscr{J}$ by Lemma~\ref{lem:dedef}.

The following construction was first carried out in \cite{LS}.
Let $S$ be an arbitrary Boolean inverse monoid.
Denote by $[e]$ denotes the set of all idempotents $e_{1}$ such that $e_{1} \, \mathscr{D} \, e$. 
Thus $\mathsf{E}(S)/\mathscr{D}$ is just the set $\{[e] \colon e \in \mathsf{E}(S) \}$.
Define a {\em partially defined addition} on this set by
$$[e] \oplus [f] = [e' \vee f']$$ 
if $e \, \mathscr{D} \, e'$ and $f \, \mathscr{D} \, f'$ and $e'$ and $f'$ are orthogonal
and undefined otherwise.
Put $\mathbf{0} = [0]$ and $\mathbf{1} = [1]$.
Define $\mathsf{Int}(S) = \mathsf{E}(S)/\mathscr{D}$ with the above partial addition.
This structure is called the {\em partial type monoid};
Wehrung \cite[Definition 4.1.3]{Wehrung} calls this the {\em type interval}.

Now, assume that $S$ is a Foulis monoid,
and use the construction of \cite{LS}.
In a Foulis monoid, the set $[e]$ consists of all idempotents $e'$ where $SeS = Se'S$.
Define a unary operation by $\neg [e] = [\bar{e}]$,  well-defined by Proposition~\ref{prop:factorizable}.
Define
$$[e] \wedge [f] = [i]$$ 
precisely when  
$SeS \cap SfS = SiS$.
Define a binary operation on $\mathsf{Int}(S)$ by
$$[e] + [f] = [e] \oplus ([\bar{e}] \wedge [f]).$$
Put $\mathsf{L}(S)$ equal to the set  $\mathsf{E}(S)/\mathscr{D}$ 
equipped with the two operations defined above together with $\mathbf{0}$ and $\mathbf{1}$.
We have the following.

\begin{theorem}\label{thm:alls-well} 
If $S$ is a Foulis monoid, then $\mathsf{L}(S)$, with the addition defined above, is an MV-algebra.
\end{theorem}

The following was proved as \cite[Proposition 2.12]{LS}.

\begin{lemma}\label{lem:order-order} Let $S$ be a Foulis monoid.
Then $[e] \leq [f]$ in the MV-algebra $\mathsf{L}(S)$ if and only if $e \preceq f$ in $S$.
Observe that $[e] < [f]$ in the MV-algebra $\mathsf{L}(S)$ if and only if $e \prec f$ in $S$.
\end{lemma}

The following lemma simply reassures us that, in the case of Foulis monoids,
we really are extending the partially defined addition. 

\begin{lemma}\label{lem:chelsea} Let $S$ be a Foulis monoid.
Suppose that $e \perp f$.
Then $[e] \oplus [f] = [e \vee f]$.
\end{lemma}
\begin{proof} By definition $[e] \oplus [f] = [e] + ([\bar{e}] \wedge [f])$.
Since $e \perp f$, it follows that $f \leq \bar{e}$.
Thus $[\bar{e}] \wedge [f] = [f]$.
\end{proof}

\begin{example}{\em The monoids $\mathcal{I}_{n}$ are Foulis monoids by 
Lemma~\ref{lem:metal} and Lemma~\ref{lem:jade}.
In this case, $\mathsf{L}(\mathcal{I}_{n}) \cong L_{n+1}$,
since $\mathcal{I}_{n}$ contains $(n+1)$ $\mathscr{D}$-classes of idempotents. 
Thus the finite symmetric inverse monoids are classified by their associated MV-algebras.
If $S$ is a finite matricial Boolean inverse monoid, then its associated MV-algebra
is a finite MV-algebra and every finite MV-algebra arises in this way.
The finite matricial Boolean inverse monoids are classified by their associated MV-algebras.
See \cite[Theorem 2.14]{LS}.}
\end{example}
 
A subset $F \subseteq \mathsf{E}(S)$, where $0 \in F$, is called a {\em conjugate-closed additive ideal}
if it is an additive ideal of $\mathsf{E}(S)$ closed under conjugation by elements of $S$.

\begin{lemma}\label{lem:pd}  Let $S$ be a Boolean inverse monoid.
\begin{enumerate}

\item There is an order-isomorphism between the poset of additive ideals of $S$ and
the poset of conjugate-closed additive ideals of $\mathsf{E}(S)$.

\item Let $S$ be a Foulis monoid.
Then there is an order-isomorphism between the poset of additive ideals of $S$
and the poset of MV-algebra ideals of $\mathsf{L}(S)$.

\end{enumerate}
\end{lemma}
\begin{proof} (1) The map $\mathsf{E}$ takes an additive ideal $I$ to the conjugate-closed additive ideal
$\mathsf{E}(I) \subseteq \mathsf{E}(S)$.
Observe that $\mathsf{E}$ is order-preserving.
Conversely, if $F$ is a conjugate-closed additive ideal of $\mathsf{E}(S)$,
then $(SFE)^{\vee}$ is an additive ideal of $S$.
Again, this construction is order-preserving.
These two constructions are mutually inverse:
if $I$ is an additive ideal of $S$ then
$I = (S\mathsf{E}(I)S)^{\vee}$,
whereas, if 
$F \subseteq \mathsf{E}(S)$ is a conjugate-closed additive ideal,
then 
$\mathsf{E}((SFS)^{\vee}) = F$.

(2) By part (1) above,  there is an order-isomorphism between
the additive ideals of $S$ and the conjugate-closed additive ideals of $\mathsf{E}(S)$.
We shall now show that there is an order-isomorphism between the 
conjugate-closed additive ideals of $\mathsf{E}(S)$ and the MV-algebra ideals of $\mathsf{L}(S)$.

Let $F$ be a conjugate-closed additive ideal of $\mathsf{E}(S)$.
Define $[F] = \{ [e] \colon e \in F \}$.
It remains to show that $[F]$ is an ideal in $\mathsf{L}(S)$.
Certainly, $[F]$ contains $[0]$.
It is also an order-ideal by Lemma~\ref{lem:order-order}.
It remains to show that $[F]$ is closed under $+$.
Let $[e], [f] \in [F]$.
So, $e,f \in F$.
Suppose that $S\bar{e}S \cap SfS = SiS$.
Then $i \preceq f$.
It follows that $i \in F$.
Observe that $i \preceq \bar{e}$.
It follows that $[e] + [f] = [e] \oplus [i] = [e \vee i']$
where $i' \perp i$ and $i \, \mathscr{D} \, i'$.
It follows that $i' \in F$ and so $e \vee i' \in F$.
We have therefore proved that $[F]$ is closed under $+$.
If $F_{1} \subseteq F_{2}$ then $[F_{1}] \subseteq [F_{2}]$
with equality if and only if $F_{1} = F_{2}$.
Thus we have an order-preserving map given by $F \mapsto [F]$.
Now let $\mathcal{I} \subseteq \mathsf{L}(S)$ be an ideal.
Define 
$$\mathcal{E}(\mathcal{I}) = \{e \in \mathsf{E}(S) \colon [e] \in \mathcal{I}\}.$$
To prove that $\mathcal{E}(\mathcal{I})$ is a conjugate-closed additive ideal of $\mathsf{E}(S)$,
there are three conditions to check, each of which is easy:
it is an order-ideal;
it is closed under orthogonal binary joins and so it is closed under arbitrary binary joins;
it is closed under the $\mathscr{D}$-relation.
Suppose that $\mathcal{I}_{1} \subseteq \mathcal{I}_{2}$.
Then $\mathcal{E}(\mathcal{I}_{1}) \subseteq \mathcal{E}(\mathcal{I}_{2})$
with equality if and only if $\mathcal{I}_{1} = \mathcal{I}_{2}$.
Thus we have an order-preserving map given by $\mathcal{I} \mapsto \mathcal{E}(\mathcal{I})$.
It is now routine to check that the maps 
$F \mapsto [F]$ and $\mathcal{I} \mapsto \mathcal{E}(\mathcal{I})$
 are mutually inverse.
 
 If we combine what we have found above with the result of part (1),
 then we have proved our order-isomorphism.
\end{proof}

The following is immediate by the above and the only case we need.

\begin{proposition}\label{prop:green} Let $S$ be a Foulis monoid.
Then $S$ is $0$-simplifying if and only if $\mathsf{L}(S)$ is simple.
\end{proposition}

We shall now prove a few results about Foulis monoids in which 
$S/\mathscr{J}$ is linearly ordered.

\begin{lemma}\label{lem:muffin} Let $S$ be a Foulis monoid in which $S/\mathscr{J}$ is linearly ordered. 
If $e$ and $f$ are any idempotents in $S$ then 
$[e] + [f] = [1]$ if and only if  $\bar{e} \preceq f$.  
\end{lemma} 
\begin{proof}
Suppose, first, that $\bar{e} \preceq f$.  
By Lemma~\ref{lem:order-order},
$[\bar{e}] \leq [f]$ and so $[\bar{e}] \wedge [f] = [\bar{e}]$.
It follows, from the definition, that  $[e] + [f] = [e] \oplus [\bar{e}] = [1]$.
Conversely, let $[e] + [f] = [1]$.
There are two possibilities:
either 
$SfS \subset S\bar{e}S$
or
$S\bar{e}S \subseteq SfS$.  
We shall rule out the former.
Suppose that $SfS \subset S\bar{e}S$.
Then $f \preceq \bar{e}$.
And so 
$f \,\mathscr{D}\, f' \leq \bar{e}$.
Observe that $f' \perp \bar{e}$.
By definition, $[e] + [f] = [e] \oplus [f'] = [e \vee f']$.
By assumption, $e \vee f' \, \mathscr{D} \, 1$.
By Lemma~\ref{lem:fac-is-df},
it follows that
$e \vee f' = 1$. 
We also have that $f' e = 0$,
and so $f' = \bar{e}$.
Whence $SfS = Sf'S = S\bar{e}S$.
We have proved that $S\bar{e}S = SfS$, which is a contradiction. 
\end{proof}

The following is an immediate consequence of the lemma above.

\begin{corollary}\label{cor:sponge}
Let $S$ be a Foulis monoid in which $S/\mathscr{J}$ is linearly ordered.
Then $[e] + [f] < [1]$ if and only if $f \prec \bar{e}$.
\end{corollary}

We can now prove what will be an important result.

\begin{lemma}\label{lem:cheese-cake}
Let $S$ be a Foulis monoid in which $S/\mathscr{J}$ is linearly ordered.
\begin{enumerate}
\item $2[e] = [f] < [1]$ if and only if there exist orthogonal idempotents $f_{1}$ and $f_{2}$ where $f \, \mathscr{D} \, f_{1} \vee f_{2}$,
$e \, \mathscr{D} \, f_{1}$, $e \,\mathscr{D} \, f_{2}$, and $f \prec 1$.

\item  $n[e] = [f] < [1]$ if and only if there exist orthogonal idempotents $e_{1}, \ldots, e_{n}$
where
$f \, \mathscr{D} \, e_{1} \vee \ldots \vee  e_{n}$,
$e \, \mathscr{D} \, e_{i}$ for each $i$,
and $f \prec 1$.
\end{enumerate}
\end{lemma} 
\begin{proof} (1) 
Suppose that 
there exist orthogonal idempotents $f_{1}$ and $f_{2}$ where $f \, \mathscr{D} \, f_{1} \vee f_{2}$,
$f_{1} \perp f_{2}$, $e \, \mathscr{D} \, f_{1}$, $e \,\mathscr{D} \, f_{2}$, and $f \prec 1$.
Then 
$$2[e] = [e] + [e] = [f_{1}] + [f_{2}] = [f_{1} \vee f_{2}] = [f] < [1].$$
Conversely, let $2[e] = [f] < [1]$.
Thus $[e] + [e] = [f] < [1]$.
By Corollary~\ref{cor:sponge}, $e \prec \bar{e}$.
There therefore exists an idempotent $f_{2}$ such that
$e \, \mathscr{D} \, f_{2} \leq \bar{e}$.
Put $f_{1} = e$.
Then $f_{1} \perp f_{2}$
and $f \, \mathscr{D} \, f_{1} \vee f_{2}$, as required.

(2) Assume first that there exist orthogonal idempotents $e_{1}, \ldots, e_{n}$ such that
$f \prec 1$, $f \, \mathscr{D} \, e_{1} \vee \ldots \vee e_{n}$ and $e \, \mathscr{D} \, e_{i}$ for each $i$.
Then $[f] = [e_{1} \vee \ldots \vee e_{n}] = [e_{1}] + \ldots + [e_{n}] = n[e]$
where we have used Lemma~\ref{lem:chelsea}.
Thus $n[e] = [f] < [1]$.

To prove the converse, we shall use induction.
The case where $n = 2$ was handled in part (1).
Suppose that  $n[e] = [f] < [1]$.
Then $n[e] = [e] + (n-1)[e]$.
Let $(n-1)[e] = [g]$.
By induction, there exist idempotents $g_{2}, \ldots, g_{n}$ such that $g \, \mathscr{D} \, g_{2} \vee \ldots \vee g_{n}$,
$e \, \mathscr{D} \, g_{i}$ where $2 \leq i \leq n$,
and $g_{k} \perp g_{l}$ where $k \neq l$ and $2 \leq k, l \leq n$.
This leaves us with $[e] + [g] = [f] < [1]$.
By Lemma~\ref{lem:muffin}, $g \prec \bar{e}$.
Thus $[e] + [g] = [e \vee g]$ 
since $g \perp e$.
It follows that $f \, \mathscr{D} \, e \vee g$.
But  $g \, \mathscr{D} \, g_{2} \vee \ldots \vee g_{n}$.
Let $g \stackrel{y}{\leftarrow}  g_{2} \vee \ldots \vee g_{n}$.
We now use Lemma~\ref{lem:jefferson}. 
Put $e_{i} = \mathbf{r}(yg_{i})$.
We therefore have that $g = e_{2} \vee \ldots \vee e_{n}$.
Put $e_{1} = e$.
Then the $e_{j}$ are orthogonal to each other,
$f \, \mathscr{D} \,  e_{1} \vee e_{2} \vee \ldots \vee e_{n}$,
$e \, \mathscr{D} \, e_{j}$,
and $f \prec 1$.
\end{proof}

\subsection{Multi-infinitesimals}

Recall that with respect to the restricted product, an inverse semigroup is a groupoid.
We shall be interested in simplicial groupoids \cite{Higgins} therein, although we shall handle them in a particular way.

We begin with a definition which is nothing other than a version of the multisections of \cite{Nek}.
Let $n \geq 2$.
By a {\em multi-infinitesimal} or an {\em $n$-infinitesimal} $\mathbf{a} = (a_{n}, \ldots, a_{1})$ in a Boolean inverse monoid $S$
is meant a sequence of $n$ elements of $S$ such that 
$$\xymatrix{
e_{n+1} & \ar[l]_{a_{n}} e_{n} & \ar[l]_{a_{n-1}} e_{n-1} & \ar@{--}[l]   e_{3} &  \ar[l]_{a_{2}} e_{2} & \ar[l]_{a_{1}} e_{1}
}
$$
(from the picture $\mathbf{d}(a_{i+1}) = \mathbf{r}(a_{i})$)
where the $n+1$ idempotents $e_{1}, \ldots, e_{n+1}$ are orthogonal.
The class of $2$-infinitesimals were used in \cite{Lawson2016}.

\begin{remark}\label{rem:barry}{\em 
Define a {\em circular multi-infinitesimal} to be a sequence of orthogonal idempotents $e_{1}, \ldots, e_{n+1}$
linked by elements $e_{i} \stackrel{a_{i}}{\rightarrow} e_{i+1}$, where $1 \leq i \leq n$,
and $e_{n+1} \stackrel{a_{n+1}}{\rightarrow} e_{1}$
such that $a_{1} \ldots a_{n+1} = e_{1}$.
The elements $a_{1}, \ldots, a_{n+1}$ are called the {\em components} of the circular multi-infinitesimal.
If we erase any one component of a circular multi-infinitesimal, we get a multi-infinitesimal.
On the other hand, given an $n$-infinitesimal, 
adjoining the element 
$(a_{n} \ldots a_{1})^{-1}$ 
converts it into a circular multi-infinitesimal.
Thus multi-infinitesimals and circular multi-infinitesimals are different ways
of viewing the same object.

Given a circular multi-infinitesimal,
connect any two vertices by an edge labelled by an appropriate product of the components.
In this way, we obtain a simplical groupoid \cite{Higgins}
$\mathsf{grpd}(\mathbf{a})$
with
the additional property that the identities form an orthogonal set.
It is clear that any such simplicial groupoid gives rise to a circular multi-infinitesimal.
}
\end{remark}

\begin{example}{\em Let $\mathbf{a}$ be the following $2$-infinitesimal:

$$\xymatrix{
e_{3} &  \ar[l]_{a_{2}} e_{2} & \ar[l]_{a_{1}} e_{1}
}
$$
\noindent
If we adjoin $a_{3} = (a_{2}a_{1})^{-1}$, then we get the following circular multi-infinitesimal

$$\xymatrix{& &e_{1} \ar[dr]^{a_{1}} & \\
&e_{3} \ar[ur]^{a_{3}} & &e_{2} \ar[ll]_{a_{2}}
}
$$
\noindent
If we adjoin the inverses of the components, we get the simplicial groupoid $\mathsf{grpd}(\mathbf{a})$:

$$\xymatrix{& &e_{1} \ar[dr] \ar@/^/[dl] & \\
&e_{3} \ar@/_/[rr] \ar[ur] & &e_{2} \ar[ll] \ar@/^/[ul]
}
$$
}
\end{example}

We now give names to the ingredients of a multi-infinitesimal.
If $\mathbf{a}$ is a multi-infinitesimal then the entries $a_{i}$ are called its {\em components}.
The components $a_{i+1}$ and $a_{i}$ are said to be {\em adjacent}.
The idempotents $e_{1},\ldots, e_{n+1}$ are said to be {\em in} the multi-infinitesimal $\mathbf{a}$.
Observe that all the idempotents in a multi-infinitesimal are $\mathscr{D}$-related to each other.
The {\em source} of $\mathbf{a}$ is the idempotent $e_{1}$,
denoted by $\alpha (\mathbf{a})$, 
and the {\em target} of $\mathbf{a}$ is the idempotent $e_{n+1}$,
denoted by $\omega (\mathbf{a})$.
The join of all of the idempotents in $\mathbf{a}$ is called the {\em extent} of the multi-infinitesimal, denoted by $\mathbf{e}(\mathbf{a})$.
Observe that each component of a multi-infinitesimal really is an infinitesimal.

\begin{remark}{\em If $\mathbf{a}$ is a multi-infinitesimal
then we can {\em never} have $\alpha (\mathbf{a}) = 1$.}
\end{remark}

Let $\mathbf{a} = (a_{n}, \ldots, a_{1})$ be an $n$-infinitesimal
in a Boolean inverse monoid $S$ with extent $e$.
Put $G = \mathsf{grpd}(\mathbf{a})$.
This is a connected principal groupoid with $n+1$ identities.
Adjoin the zero of $S$ to $G$ to get the inverse subsemigroup $G^{0}$.
Now put $T = \mathsf{inv}(\mathbf{a}) = (G^{0})^{\vee}$.
Then $T$ is a Boolean inverse subalgebra of $eSe$.

Multi-infinitesimals are intimately connected with subalgebras of $S$ which are isomorphic to finite symmetric inverse monoids as we now show.

\begin{lemma}\label{lem:jordan} 
Let $S$ be a Boolean inverse monoid and let $e$ be an idempotent in $S$
and let $n \geq 1$.
The following are equivalent:
\begin{enumerate}
\item There is an additive embedding of $\mathcal{I}_{n+1}$ into $eSe$.
\item There is an $n$-infinitesimal $\mathbf{a}$ such that $e = \mathbf{e}(\mathbf{a})$.
\end{enumerate}
\end{lemma}
\begin{proof} (1)$\Rightarrow$(2).
Let $S$ be a Boolean inverse monoid containing the idempotent $e$.
Suppose that $T$ is a subalgebra of $eSe$, where $T$ is isomorphic to the finite symmetric inverse monoid $\mathcal{I}_{n+1}$.
Then $T$ contains copies (under the isomorphism) of the $n$ partial bijections
$i \mapsto i+1$ where $i = 1, \ldots, n$,
each of which is an atom of $T$, and all together form a multi-infinitesimal whose extent is $e$.

(2)$\Rightarrow$(1).
Let $\mathbf{a}$ be an $n$-infinitesimal.
Then $T = \mathsf{inv}(\mathbf{a})$
is a Boolean inverse subalgebra of $eSe$ isomorphic to $\mathcal{I}_{n+1}$. 
\end{proof}

We define five operations on multi-infinitesimals:
truncation, splicing, restriction, translation, and product.
We shall use these operations in proving our main theorem.

\begin{center}
{\bf Truncation}
\end{center}

This is a simple operation.
If $\mathbf{a} = (a_{n}, \ldots, a_{1})$ is an $n$-infinitesimal
then we may obtain an $(n-1)$-infinitesimal $\mathbf{a}'$, called the {\em elementary truncation} of $\mathbf{a}$,
by defining $\mathbf{a}' = (a_{n}, \ldots, a_{2})$.
The extent of an elementary truncation is strictly less than the original extent. 
But we have that $\omega (\mathbf{a}') = \omega (\mathbf{a})$.
We may apply any number of elementary truncations to a multi-infinitesimal
to obtain a multi-infinitesimal until we reach an individual infinitesimal when we stop.
In such a sequence, the extents are strictly decreasing.
We shall also call any multi-infinitesimal $\mathbf{b}$ a {\em truncation} of $\mathbf{a}$ if it arises by means of a sequence of elementary trunctions applied to the multi-infinitesimal $\mathbf{a}$.

\begin{center}
{\bf Splicing}
\end{center}

The next operation is also very simple, and the proofs are easy.

\begin{lemma}[Splicing]\label{lem:splice}
Let $\mathbf{a}$ be an $m$-infinitesimal such that $\alpha (\mathbf{a}) = e$,
and 
let $\mathbf{b}$ be an  $n$-infinitesimal such that $\omega (\mathbf{b}) = f$
and let $x$ be any element $e \stackrel{x}{\leftarrow} f$.
and, suppose also, that $\mathbf{e}(\mathbf{a}) \perp \mathbf{e}(\mathbf{b})$.
Then there is a multi-infinitesimal
$\mathbf{a}x\mathbf{b} = (a_{m}, \ldots, a_{1}, x, b_{n}, \ldots, b_{1})$
such that 
\begin{itemize}
\item $\mathbf{e}(\mathbf{a}x\mathbf{b}) = \mathbf{e}(\mathbf{a}) \vee \mathbf{e}(\mathbf{b})$.
\item $\alpha (\mathbf{a}x\mathbf{b}) = \alpha (\mathbf{b})$.
\item $\omega (\mathbf{a}x\mathbf{b}) = \omega (\mathbf{a})$.
\item $\mathsf{inv}(\mathbf{a}) \cup \mathsf{inv}(\mathbf{b}) \subseteq \mathsf{inv}(\mathbf{a}x\mathbf{b})$.
\end{itemize}
\end{lemma}

\begin{center}
{\bf Restriction}
\end{center}

The next operation is slightly more involved.

\begin{lemma}[Restriction]\label{lem:restriction} Let $S$ be a Boolean inverse monoid.
Let $\mathbf{a}$ be an $n$-infinitesimal and let $f \leq \alpha (\mathbf{a})$.
Then we may construct the $n$-infinitesimal $(\mathbf{a}|f)$, called the 
{\em restriction of $\mathbf{a}$ to $f$},
which has components $a_{i+1}\mathbf{r}(a_{i} \ldots a_{1}f)$
where $1 \leq i \leq n$ such that
\begin{itemize}
\item $\alpha (\mathbf{a}|f) = f$. 
\item $\mathbf{e}(\mathbf{a}|f) \leq \mathbf{e}(\mathbf{a})$.
\end{itemize}
\end{lemma}
\begin{proof} By construction, the domains and ranges of the new elements match up.
Each component of  $(\mathbf{a}|f)$ has the form $a_{i}k \leq a_{i}$ where $k$ is an idempotent.
It follows that $\mathbf{r}(a_{i}k) \leq \mathbf{r}(a_{i})$
and $\mathbf{d}(a_{i}k) \leq \mathbf{d}(a_{i})$.
Thus the idempotents that arise are certainly orthogonal to each other.
\end{proof}

We shall need the following lemma later.

\begin{lemma}\label{lem:cary}
Let $S$ be a Boolean inverse monoid.
Let $\mathbf{a}$ be an $n$-infinitesimal and let $e, f \leq \alpha (\mathbf{a})$
where $e \perp f$.
Then any idempotent in $(\mathbf{a}|e)$ is orthogonal to any idempotent in $(\mathbf{a}|f)$.
\end{lemma}
\begin{proof} That both $(\mathbf{a}|e)$ and $(\mathbf{a}|f)$ are multi-infinitesimals
is immediate by Lemma~\ref{lem:restriction}.
That every idempotent in $(\mathbf{a}|e)$ 
is orthogonal to every idempotent in $(\mathbf{a}|f)$ 
follows from the fact that $e \perp f$ and $\mathbf{a}$ is a multi-infinitesimal.
\end{proof}

\begin{center}
{\bf Translation}
\end{center}

To define our next operation, we need a lemma first.

\begin{lemma}\label{lem:jefferson} Let $S$ be a Boolean inverse monoid.
Suppose that $x \in S$ is such that $\mathbf{d}(x) = \bigvee_{i=1}^{n}e_{i}$.
\begin{enumerate}

\item Then $x$ is a join of the elements $xe_{i}$ and $\mathbf{r}(x) = \bigvee_{i=1}^{n} \mathbf{r}(xe_{i})$.

\item If the idempotents $e_{i}$ are orthogonal, so too are the idempotents  $\mathbf{r}(xe_{i})$.

\item If the idempotents 
$$e_{1}, \ldots, e_{n}$$ 
are $\mathscr{D}$-related, so too are the idempotents 
$$\mathbf{r}(xe_{1}), \ldots, \mathbf{r}(xe_{n}).$$

\end{enumerate}
\end{lemma}
\begin{proof} (1) The elements  $xe_{i}$ are pairwise compatible and so have a join
$y \leq x$. But $\mathbf{d}(y) = \mathbf{d}(x)$ and so $x = \bigvee_{i=1}^{n} xe_{i}$.
However, $\mathbf{r}(x) = \bigvee_{i=1}^{n} \mathbf{r}(xe_{i})$.

(2) This follows from the fact that the idempotents $e_{i}$  are pairwise orthogonal.

(3) If $e_{i}$ is $\mathscr{D}$-related to $e_{i+1}$ then $\mathbf{r}(xe_{i})$ is  $\mathscr{D}$-related to $\mathbf{r}(xe_{i+1})$.
To see why, suppose that $e_{i} \stackrel{a}{\rightarrow} e_{i+1}$.
Then we have the following diagram:
$$\xymatrix{ 
e_{i }\ar[d]_{xe_{i}}   \ar[r]^{a} & e_{i+1} \ar[d]^{xe_{i+1}} &  \\
\mathbf{r}(xe_{i} ) \ar[r] & \mathbf{r}(xe_{i+1})
}
$$
This immediately implies that 
$\mathbf{r}(xe_{i})\, \mathscr{D} \,\mathbf{r}(xe_{i+1})$.
\end{proof}

We may now describe our next operation.

\begin{lemma}[Translation]\label{lem:adams} Let $S$ be a Boolean inverse semigroup.
Let $x \in S$ be an element and let $\mathbf{a}$ be an $n$-infinitesimal, 
with components $a_{1}, \ldots, a_{n}$
and containing the idempotents $e_{1}, \ldots, e_{n+1}$,
such that $\mathbf{d}(x) = \mathbf{e}(\mathbf{a})$.
Then we can define a new $n$-infinitesimal, denoted by $x \cdot \mathbf{a}$, 
whose components are $(xe_{i+1})a_{i}(xe_{i})^{-1}$ such that
\begin{itemize}
\item $\mathbf{e}(x \cdot \mathbf{a}) = \mathbf{r}(x)$.
\item $\alpha (x \cdot \mathbf{a}) \, \mathscr{D} \, \alpha (\mathbf{a})$
\item $\omega (x \cdot \mathbf{a}) \, \mathscr{D} \, \omega (\mathbf{a})$.
\end{itemize}
\end{lemma}
\begin{proof} We use Lemma~\ref{lem:jefferson}.
By assumption, $\mathbf{d}(x)$ is a join of the orthogonal idempotents $e_{1}, \ldots , e_{n+1}$.
It follows that $x$ is a join of the orthogonal set $xe_{1}, \ldots, xe_{n+1}$.
Thus $\mathbf{r}(x)$ is a join of the orthogonal idempotents $\mathbf{r}(xe_{1}), \ldots, \mathbf{r}(xe_{n+1})$.
The result now follows.
\end{proof}

\begin{center}
{\bf Product}
\end{center}

Let $S$ be a Boolean inverse monoid
which we regard as a groupoid with respect to the restricted product.
By an {\em $n \times m$ grid} over $S$, 
we mean a rectangular $n \times m$ array of squares,
the side of each square is labelled by an element of $S$,
and the squares commute (in the groupoid),
each vertex is labelled by an idempotent of $S$, 
and the idempotents are pairwise orthogonal.
For clarity, there are $(m+1) \times (n+1)$ vertices.
The word `row' will refer to a horizontal sequence of squares.
 If $\mathbf{g}$ is a grid then its extent $\mathbf{e}(\mathbf{g})$
is the join of all the orthogonal idempotents at the vertices.
Every grid determines a simplicial groupoid and so
we can define $\mathsf{inv}(\mathbf{g})$, a copy of a symmetric inverse monoid.
If the grid has $s$ vertices then $\mathsf{inv}(\mathbf{g})$ is isomorphic to $\mathcal{I}_{s}$.
Given any grid $\mathbf{g}$ we denote by $\alpha (\mathbf{g})$
the idempotent in the top right hand corner,
whereas $\omega (\mathbf{g})$ is the idempotent in the bottom left hand corner.

\begin{example}\label{ex:anchorite}
{\em  The following is an example of the form taken by a $2 \times 2$ grid:

$$\xymatrix{
\bullet \ar[d] &   \ar[l] \ar[d] \bullet & \ar[l] \ar[d] \bullet\\
\bullet \ar[d] &  \ar[l] \ar[d] \bullet & \ar[l] \ar[d] \bullet\\
\bullet &  \ar[l] \bullet & \ar[l] \bullet\\
}
$$
\noindent
Each arrow is labelled by an element of $(S,\cdot)$ (regarded as a groupoid),
domains and ranges in the groupoid match up,
and squares commute with the convention that in travelling against the arrow
we take the inverse of the element.}
\end{example}

\begin{lemma}[Product]\label{lem:franklin} Let $S$ be a Boolean inverse monoid.
Let $\mathbf{a}$ be an $n$-infinitesimal and let $\mathbf{b}$ be an $m$-infinitesimal
such that $\alpha (\mathbf{a}) = \mathbf{e}(\mathbf{b})$,
and let the idempotents contained in $\mathbf{b}$
be $e_{1}, \ldots, e_{m+1}$.
Define a grid, denoted by $\mathbf{a} \times \mathbf{b}$, as follows:
\begin{itemize}
\item  The top of the grid, which we shall call the {\em spine}, is $\mathbf{b}$.
\item  Form the restrictions
$$(\mathbf{a}|e_{m+1}), (\mathbf{a}|e_{n}), \ldots, (\mathbf{a} | e_{1})$$
by Lemma~\ref{lem:restriction}.
\item Using $\mathbf{b}$ as the backbone, attach each restriction 
$(\mathbf{a}|e_{i})$ to the idempotent $e_{i}$ in the backbone
to obtain a `comb'.
\item Fill-in any horizontal lines, starting at the top, using commutativity in the groupoid.
\end{itemize}
The result is a grid, denoted by $\mathbf{a} \times \mathbf{b}$,
such that 
\begin{itemize}
\item $\mathbf{e} (\mathbf{a} \times \mathbf{b}) = \mathbf{e}(\mathbf{a})$.
\item $\mathsf{inv}(\mathbf{a}), \mathsf{inv}(\mathbf{b}) \subseteq \mathsf{inv}(\mathbf{a} \times \mathbf{b})$.
\item Any idempotent in $\mathbf{a} \times \mathbf{b}$ is $\mathscr{D}$-related to any idempotent in $\mathbf{b}$.
\end{itemize}
\end{lemma}
\begin{proof}
That $\mathbf{a} \times \mathbf{b}$ really is a grid
follows by Lemma~\ref{lem:restriction} and Lemma~\ref{lem:cary}.

We now calculate the joins of the elements (pointing downwards) in each row.
In the first row, we have that 
$$\bigvee_{i=1}^{m+1} a_{1}e_{i} = a_{1}.$$
In row $j \geq 2$,
we have to calculate
$$\bigvee_{i=1}^{m+1} a_{j}\mathbf{r}(a_{j-1} \ldots a_{2}a_{1}e_{i}),$$
but
$$\bigvee_{i=1}^{m+1} a_{j}\mathbf{r}(a_{j-1} \ldots a_{2}a_{1}e_{i}) = a_{j}\left(\bigvee_{i=1}^{m+1} \mathbf{r}(a_{j-1} \ldots a_{2}a_{1}e_{i})\right) = a_{j}$$
since
$$\bigvee_{i=1}^{m+1} \mathbf{r}(a_{j-1} \ldots a_{2}a_{1}e_{i}) = a_{j-1} \ldots a_{1} \left( \bigvee_{i=1}^{m+1} e_{i} \right) a_{1}^{-1} \ldots a_{j-1}^{-1} 
= \mathbf{r}(a_{j-1}) 
= \mathbf{d}(a_{j}).$$
It now follows that  
$\mathbf{e} (\mathbf{a} \times \mathbf{b}) = \mathbf{e}(\mathbf{a})$
and
$\mathsf{inv}(\mathbf{a}), \mathsf{inv}(\mathbf{b}) \subseteq \mathsf{inv}(\mathbf{a} \times \mathbf{b})$
\end{proof}

\subsection{Finite lists of multi-infinitesimals}

In this section, we shall show how to handle matricial subalgebras.

Let $T = T_{1} \times T_{2}$,
where both $T_{1}$ and $T_{2}$ are Boolean inverse monoids,
with identities $e_{1}$ and $e_{2}$, respectively.
This is an `external' construction.
We now render it `internal'.
Put $T_{1}' = T_{1} \times \{0\}$ and $T_{2}' = \{0\} \times T_{2}$.
then $T_{1} \cong T_{1}'$ and $T_{2} \cong T_{2}'$
and $T = T_{1}' \vee T_{2}'$;
in other words, each element of $T$ is equal to a join of an element of $T_{1}'$ and an element of $T_{2}'$. 
The identity of $T$ is $(e_{1},e_{2}) = (e_{1},0) \vee (0,e_{2})$.
Observe that $(e_{1},0)T(e_{1},0) = T_{1}' \cong T_{1}$
and $(0,e_{2})T(0,e_{2}) = T_{2}' \cong T_{2}$.

More generally, 
let $S$ be a Boolean inverse monoid where $T_{1},\ldots,T_{m} \subseteq S$ are non-empty subsets
then $T_{1} \vee \ldots \vee T_{m}$ consists of all elements of $S$ of the form $t_{1} \vee \ldots \vee t_{m}$
where $t_{i} \in T_{i}$ and the join is defined.

\begin{lemma}\label{lem:tom} Let $S$ be a Boolean inverse monoid containing idempotents $e_{1}$ and $e_{2}$,
neither of which is zero and where $e_{1} \perp e_{2}$.
Suppose that 
$T_{1}$ is a subalgebra of $e_{1}Se_{1}$,
and
$T_{2}$ is a subalgebra of $e_{2}Se_{2}$.
Then
\begin{enumerate}
\item $T_{1} \vee T_{2}$ is a subalgebra of $(e_{1} \vee e_{2})S(e_{1} \vee e_{2})$.
\item $T_{1} \vee T_{2} \cong T_{1} \times T_{2}$.
\end{enumerate}
\end{lemma}
\begin{proof} (1)
Let $a \in T_{1}$ and $b \in T_{2}$.
Then $\mathbf{a} \leq e$ and $\mathbf{d}(b) \leq f$ and $e \perp f$.
It follows that $\mathbf{d}(a)\mathbf{d}(b) = 0$.
We may similarly show that $\mathbf{r}(a)\mathbf{r}(b) = 0$.
We deduce that $a \perp b$.
It follows that all joins in $T_{1} \vee T_{2}$ exist
and that $T_{1}, T_{2} \subseteq T_{1} \vee T_{2}$.
It is clear that the identity of $T_{1} \vee T_{2}$ is $e \vee f$.
The set $T_{1} \vee T_{2}$ is closed under binary compatible joins
and the semilattice of idempotents is $e^{\downarrow} \cup f^{\downarrow} = (e \vee f)^{\downarrow}$,
which is a Boolean algebra.
The idempotents in $T_{1} \vee T_{2}$ have the form $e_{1} \vee f_{1}$ where $e_{1}$ is an idempotent in $T_{1}$
and $f_{1}$ is an idempotent in $T_{2}$.
The complement in $(e \vee f)S(e \vee f)$ of the idempotent 
$e_{1} \vee f_{1}$ is
$e\overline{e_{1}} \vee f \overline{f_{1}}$. 

(2) Let $s \in T_{1} \vee T_{2}$.
Then $s = t_{1} \vee t_{2}$ where $t_{1} \in T_{1}$ and $t_{2} \in T_{2}$.
Suppose that $s = t_{1}' \vee t_{2}'$ where $t_{1}' \in T_{1}$ and $t_{2}' \in T_{2}$.
Then $se = t_{1}e = t_{1}$ and $se = t_{1}'e = t_{1}'$.
It follows that if we write $s = t_{1} \vee t_{2}$ where $t_{1} \in T_{1}$ and $t_{2} \in T_{2}$
then this is unique. We therefore define a map from $T_{1} \vee T_{2}$ to $T_{1} \times T_{2}$
by $s \mapsto (t_{1},t_{2})$. This is an isomorphism of inverse monoids.
\end{proof}

The notation we have introduced enables us now to handle matricial subalgebras efficiently;
we need only replace two factors by $m$ factors.

\begin{lemma}\label{lem:alex} 
Let $S$ be a Boolean inverse monoid.
\begin{enumerate}

\item Let $T$ be a subalgebra of $S$ where
$T \cong I_{1} \times \ldots \times I_{m}$
and each $I_{i}$ is a finite symmetric inverse monoid.
Then there are $m$ multi-infinitesimals $\mathbf{a}_{1}, \ldots, \mathbf{a}_{m}$ in $S$
such that $\mathsf{inv}(\mathbf{a}_{i}) \cong I_{i}$,
and $\mathbf{e}(\mathbf{a}_{i}) = f_{i}$,
where the idempotents $f_{1}, \ldots, f_{m}$ are orthogonal and their join is $1$.

\item Let $\mathbf{a}_{1}, \ldots, \mathbf{a}_{m}$ be $m$ multi-infinitesimals
such that $\mathbf{e}(\mathbf{a}_{i}) = f_{i}$
where $f_{1}, \ldots, f_{m}$ are $m$ orthogonal idempotents, whose join is $1$.
Then 
$$T = \mathsf{inv}(\mathbf{a}_{1}) \vee \ldots  \vee \mathsf{inv}(\mathbf{a}_{m})$$
is a well-defined subalgebra of $S$ isomorphic to a direct product of $m$ finite symmetric inverse monoids.

\end{enumerate}
\end{lemma} 
\begin{proof} (1) Let $I = I_{1} \times \ldots \times I_{m}$
be a direct product of $m$ finite symmetric inverse monoids.
Let the identity of $I_{i}$ be $e_{i}$.
Define $e_{i}'$ to be the element of $I$ that has $e_{i}$ in the $i$th position
and $0$s elsewhere.
Then $e_{1}', \ldots, e_{m}'$ are idempotents of $I$ and mutually orthogonal.
We suppose that $T \cong I$ for some isomorphism.
Denote by $f_{i}$ the idempotent of $T$ that is mapped to $e_{i}'$ by the isomorphism.
Because $T$ is a subalgebra of $S$ the join of the $f_{i}$ is $1$.
The atoms of $I_{i}$, regarded as elements of $I$, give rise to a multi-infinitesimal $\mathbf{b}_{i}$ 
in $I$. Under the isomorphism, this gives rise to a multi-infinitesimal $\mathbf{a}_{i}$,
the extent of which is $f_{i}$. 

(2) Each of $T_{i} = \mathsf{inv}(\mathbf{a}_{i})$ is a finite symmetric inverse monoid 
by Lemma~\ref{lem:jordan}, and subalgebra of $f_{i}Sf_{i}$.
It follows by Lemma~\ref{lem:tom} and induction that
$T = \mathsf{inv}(\mathbf{a}_{1}) \vee \ldots  \vee \mathsf{inv}(\mathbf{a}_{m})$
is a subalgebra of $S$.
It further follows that $T$ is isomorphic to a finite direct product of finite symmetric inverse monoids
by Lemma~\ref{lem:tom} and induction.
\end{proof}

\subsection{Characterization of UHF monoids}

In formulating our main theorem, we have taken on board the suggestion of the referee and so relied
more upon the theory of MV-algebras.

The following definition is taken from \cite{KLLR}.
Let $S$ be a Boolean inverse monoid.
A function $\nu \colon \mathsf{E}(S) \rightarrow [0,1]$ is said to be a {\em (normalized) invariant mean}
if it satisfies three conditions:
\begin{enumerate}
\item $\nu (s^{-1}s) = \nu (ss^{-1})$ for all $s \in S$.
\item If $e$ and $f$ are orthogonal idempotents then $\nu (e \vee f) = \nu (e) + \nu (f)$.
\item $\nu (1) = 1$.
\end{enumerate}
We shall omit the word `normalized' in what follows.
Observe that $\nu (\bar{e}) = 1 - \nu (e)$.

The following was proved in \cite{LS}, but we provide proofs anyway because of its importance.

\begin{lemma}\label{lem:nim} \mbox{}
\begin{enumerate}
\item The finite symmetric inverse monoid $S = \mathcal{I}_{n}$ has exactly one invariant mean $\nu_{S}$ which takes only rational values.
\item If $\alpha \colon S \rightarrow T$ is a morphism between finite symmetric inverse monoids
then $\nu_{T}(\alpha (e)) = \nu_{S}(e)$. 
\end{enumerate}
\end{lemma}
\begin{proof}
(1) We state the result first.
Let $e$ be an idempotent in $\mathcal{I}_{n}$.
Then $e = 1_{A}$ where $A$ is a subset of $\{1, \ldots, n \}$.
Then $\nu (e) = \frac{\left| A \right|}{n}$, a rational number.
To prove this is correct, observe that the atomic idempotents map $i$ to $i$
where $i \in \{1, \ldots, n \}$.
Denote this atomic idempotent by $e_{i}$.
Now, observe, that $1 = \bigvee_{i=1}^{n} e_{i}$
and that $e_{i}\,  \mathscr{D} \, e_{j}$
and that every idempotent is a join of some of the atomic idempotents.
The defining properties of invariant means now deliver the result.

(2) Without loss of generality, we may assume that $S = \mathcal{I}_{m}$
and $T = \mathcal{I}_{n}$.
From $1 = \bigvee_{i=1}^{n} e_{i}$, we deduce that $1 = \alpha (1) = \bigvee_{i=1}^{n} \alpha (e_{i})$.
Since $e_{i}\,  \mathscr{D} \, e_{j}$, we have that $\alpha (e_{i}) \,  \mathscr{D} \, \alpha (e_{j})$.
The result now follows using the atomic idempotents in $T$.
\end{proof}

We now prove the easy direction of the main theorem of this section.

\begin{proposition}\label{prop:UHF-easy}
let $S$ be a countably infinite Boolean inverse monoid.
If $S$ is UHF then $S$ is an AF Foulis monoid in which $\mathsf{L}(S)$ is a rational MV-algebra.
\end{proposition}
\begin{proof}
Let $S$ be a UHF monoid.
It is an AF monoid by construction.
We prove that the poset $S/\mathscr{J}$ is linearly ordered.
Let $a$ and $b$ be non-zero elements of $S$.
Then $a,b \in S_{i}$ for some $i$.
But $S_{i}$ is a finite symmetric inverse monoid.
On $S_{i}$ the $\mathscr{J}$ relation induces a linear order by 
Lemma~\ref{lem:jade}.
Without loss of generality, we may assume that $a = xby$ where $x,y \in S_{i}$.
It follows that $SaS \subseteq SbS$ in $S$, as required.
By Theorem~\ref{them:locally-matricial}, $S$ is basic and so piecewise-factorizable by Proposition~\ref{prop:basic-properties}.
It is therefore factorizable by Lemma~\ref{lem:basic-factorizable}.
We have therefore proved that $S$ is a Foulis monoid.
We prove that UHF monoids are $0$-simplifying.
Let $I$ be a non-zero additive ideal of $S$.
We prove that $I = S$.
Let $a \in I$ be any non-zero element of $I$.
Then $a \in S_{i}$ for some $i$.
It follows that $S_{i} \cap I \neq \{0\}$
and  $S_{i} \cap I$ is a non-zero additive ideal in $S_{i}$.
But $S_{i}$ is $0$-simplifying.
It follows that $S_{i} \subseteq I$.
We can now see that $I \cap S_{j} \neq \{0\}$ for all $j \geq i$.
In every case $S_{j} \subseteq I$.
We deduce that $I = S$.
By Proposition~\ref{prop:green}, $\mathsf{L}(S)$ is a simple MV-algebra.
By Proposition~\ref{prop:simple-MV},
$\mathsf{L}(S)$ is an MV-subalgebra of $[0,1]$.
It remains to prove that $\mathsf{L}(S)$ is rational.
We prove first that $S$ is equipped with a unique invariant mean.
To begin with, we have to show that $S$ is actually equipped with an invariant mean.
Define $\nu (e) = \nu_{i}(e)$ if $e \in S_{i}$ where 
$\nu_{i}$ is the unique invariant mean on $S_{i}$
guaranteed by part (1) of 
Lemma~\ref{lem:nim}.
Suppose that $e \in S_{j}$ also.
Without loss of generality, we may assume that $S_{i} \subseteq S_{j}$.
We now apply part (2) of Lemma~\ref{lem:nim}
to deduce that $\nu$ is well-defined.
It is an invariant mean since $\nu_{i}$ is an invariant mean.
Uniqueness now follows from part (1) of Lemma~\ref{lem:nim}.
Observe that $\nu$ always takes rational values.
Now define
$\theta \colon \mathsf{L}(S) \rightarrow [0,1] \cap \mathbb{Q}$ 
by $\theta ([e]) = \nu (e)$.
This is a well-defined function since $\nu$ is an invariant mean that always
takes rational values.
We prove that  
$\theta$ is an injective homomorphism of MV-algebras.
We begin by showing that $\theta$ is a morphism of MV-algebras.
We deal with negations first.
By \cite[Lemma 2.1]{KLLR}, we have that $\nu ([\bar{e}]) = 1 - \nu ([e]) = \nu (\neg [e])$.
We calculate $\theta ([e] \oplus [f])$ and show it is equal to $\theta ([e]) \oplus \theta ([f])$.
We are working inside a Foulis monoid where $S/\mathscr{J}$ is linearly ordered.
There are therefore two cases.
Either $S\bar{e}S \subseteq SfS$ or $SfS \subseteq S\bar{e}S$.
Suppose the former.
Then $[\bar{e}] \wedge [f] = [\bar{e}]$.
It follows that $[e] \oplus [f] = [1]$.
Thus $[f] = [\bar{e}]$ and so $\theta ([e] \oplus [f]) = 1$.
On the other hand,  $\theta ([e]) \oplus \theta ([f]) = \nu (e) + (1 - \nu (e)) = 1$.
Suppose now that $SfS \subseteq S\bar{e}S$.
Then $[e] \oplus [f] = [e] + [f]$.
Thus, there is an idempotent $f'$ such that $f' \, \mathscr{D} \,  f$ where $f' \perp e$.
It follows that $[e] \oplus [f] = [e \vee f']$.
Thus  $\theta ([e] \oplus [f]) = \theta [e \vee f'] = \theta ([e]) \oplus \theta ([f]) = \nu (e) + \nu (f)$.
This is also equal to  $\theta ([e]) \oplus \theta ([f])$.
This map has to be an embedding since the kernel is an ideal of $\mathsf{L}(S)$,
which is simple.
\end{proof}

The proof of the hard direction relies on the following definition.
We say that a Boolean inverse monoid $S$ satisfies the {\em idempotent condition}
if for all non-zero idempotents $e$ and $f \neq 1$ the following holds:
$e \prec f$ implies there exist orthogonal idempotents $f_{1}, \ldots, f_{n}$
where
$f_{i} \, \mathscr{D} \, f_{j}$, for all $1 \leq i,j \leq n$,
$f \, \mathscr{D} \, f_{1} \vee \ldots \vee f_{n}$
and there exists $m < n$ such that
$e \, \mathscr{D} \, f_{1} \vee \ldots \vee f_{m}$.
It is easy to check that finite symmetric inverse monoids satisfy the idempotent condition
as well as UHF monoids.

\begin{proposition}\label{prop:stone}
Let $S$ be an 
AF Foulis monoid in which $\mathsf{L}(S)$ is a rational MV-algebra.
Then the idempotent condition holds.
\end{proposition}
\begin{proof}
Let $e_{1}$ and $f_{1}$ be non-zero idempotents of $S$ such that $e_{1} \prec f_{1}$ where $f_{1} \neq 1$.
By Lemma~\ref{lem:order-order}, we have that $[e_{1}] < [f_{1}]$ in  $\mathsf{L}(S)$.
But,  $L(S)$ is locally finite by Proposition~\ref{prop:locally-finite-MV}.
Thus $\langle [e_{1}], [f_{1}] \rangle$ is a finite MV-subalgebra of $\mathsf{L}(S)$.
We know exactly which finite MV-subalgebras can occur by Lemma~\ref{lem:zelensky}.
We now apply Lemma~\ref{lem:trump}.
There is an element $[e]$ of 
$\langle [e_{1}], [f_{1}] \rangle$ 
such that
$[e_{1}] = t[e]$ and $[f_{1}] = s[e]$ where $s > t$.
Since $[f_{1}] = s[e]$, we may apply Lemma~\ref{lem:cheese-cake} to write $f_{1} \, \mathscr{D} \, g_{1} \vee \ldots \vee g_{s}$
where the $g_{i}$ form a set of orthogonal idempotents and $e \, \mathscr{D} \, g_{i}$.
Similarly, $e_{1} \,\mathscr{D} \, h_{1} \vee \ldots \vee h_{t}$ where the $h_{j}$ form a set of orthogonal idempotents
and $e \, \mathscr{D} \, h_{j}$.
Let $h_{i} \stackrel{x_{i}}{\rightarrow} g_{i}$.
Put $x = \bigvee_{i=1}^{t}x_{i}$.
Then this element proves that 
$h_{1} \vee \ldots \vee h_{t} \, \mathscr{D} \, g_{1} \vee \ldots \vee g_{t}$.
Thus the idempotent condition holds.
\end{proof}

The reason why the idempotent condition is so important is explained by the next result;
it enables us to embed matricial subalgebras into finite subalgebras which are isomorphic to
finite symmetric inverse monoids.

\begin{proposition}\label{prop:marble} Let $S$ be a Boolean inverse monoid 
in which $\mathscr{D} = \mathscr{J}$, 
$S/\mathscr{J}$ is linearly ordered,
and the idempotent condition holds.
 Let $e$ and $f$ be orthogonal idempotents such that
 $T_{e}$ is a subalgebra of $eSe$ and $T_{f}$ is a subalgebra of $fSf$
 where both $T_{e}$ and $T_{f}$ are finite simple Boolean inverse monoids.
 Then there is a finite simple Boolean inverse subalgebra $T$ of $(e \vee f)S(e \vee f)$
 such that $T_{e}, T_{f} \subseteq T$.
 \end{proposition}
 \begin{proof} By Lemma~\ref{lem:jordan},
 we may write $T_{e} = \mathsf{inv}(\mathbf{a})$ and $T_{f} = \mathsf{inv}(\mathbf{b})$,
where $\mathbf{a}$ and $\mathbf{b}$ are multi-infinitesimals.
By assumption, $\mathbf{e}(\mathbf{a}) = e$ and  $\mathbf{e}(\mathbf{b}) = f$ where $e \perp f$.
Put $e_{1} = \alpha (\mathbf{a})$ and $f_{1} = \alpha (\mathbf{b})$.
Because $S/\mathscr{J}$ is linearly ordered there are three possibilities:
$Se_{1}S \subset Sf_{1}S$,
$Se_{1}S = Sf_{1}S$
or
$Sf_{1}S \subset Se_{1}S$.
The second possibility means that 
$e_{1} \, \mathscr{J} \, f_{1}$ 
and so, by our assumption,
$e_{1} \, \mathscr{D} \, f_{1}$.
Without loss of generality, we may therefore assume that
there are two cases:
(Case 1) $e_{1} \, \mathscr{D} \, f_{1}$;
(Case 2) $e_{1} \preceq f_{1}$.
We shall deal with these two cases separately.

{\em Proof via (Case 1).}
For this case, we do not need the idempotent condition.
Recall that all the idempotents in a multi-infinitesimal are $\mathscr{D}$-related.
Let $e_{1} \stackrel{x}{\leftarrow} \omega (\mathbf{b})$.
Thus we can `splice' the multi-infinitesimals $\mathbf{a}$ and $\mathbf{b}$ together,
by Lemma~\ref{lem:splice},
to obtain the multi-infinitesimal $\mathbf{a}x\mathbf{b}$.
Put $T = \mathsf{inv}(\mathbf{a}x\mathbf{b})$,
where $\mathbf{e}(\mathbf{a}x\mathbf{b}) = e \oplus f$.
Then $T$ is a finite simple Boolean inverse monoid and subalgebra of $(e \vee f)S(e \vee f)$.
Furthermore, $T_{e}, T_{f} \subseteq T$.
Thus we have proved the proposition in (Case 1).

{\em Proof via (Case 2).}
We can assume that $e_{1} \prec f_{1}$ and we invoke the idempotent condition.
There exist orthogonal $\mathscr{D}$-related idempotents $f_{1}',\ldots, f_{n}'$ and $m < n$ such that
$f_{1} \, \mathscr{D} \, f_{1}' \vee \ldots \vee f_{n}'$ 
and  
$e_{1} \, \mathscr{D} \, f_{1}' \vee \ldots \vee f_{m}'$.

We deal with $f_{1}$ first.
There is a multi-infinitesimal $\mathbf{c}$ 
the idempotents in which are $f_{n}',\ldots, f_{1}'$
and $\mathbf{e}(\mathbf{c}) = f_{1}' \vee \ldots \vee f_{n}'$.
It follows that 
$\mathbf{e}(\mathbf{c}) \, \mathscr{D} \, f_{1}$.
Let $\mathbf{e}(\mathbf{c}) \stackrel{x}{\rightarrow} f_{1}$.
By Lemma~\ref{lem:adams}, we may form the translated multi-infinitesimal $x \cdot \mathbf{c}$
which, 
by construction, is such that $\mathbf{e}(x \cdot \mathbf{c}) = f_{1}$.
We have that $\alpha (\mathbf{b}) = \mathbf{e}(x \cdot \mathbf{c})$,
and so we may therefore form the product $\mathbf{b} \times x \cdot \mathbf{c}$
by Lemma~\ref{lem:franklin}.
Observe that $T_{f} \subseteq \mathsf{inv}(\mathbf{b} \times x \cdot \mathbf{c})$
and $\mathbf{e}(\mathbf{b} \times x \cdot \mathbf{c}) = f$.
The point is that we have embedded $T_{f}$ in a larger finite simple Boolean inverse monoid
which is still a subalgebra of $fSf$.

We now deal with $e_{1}$.
By assumption, there is a truncation $\mathbf{c}'$ of $\mathbf{c}$
such that $e_{1} \, \mathscr{D} \, \mathbf{e}(\mathbf{c}')$.
Let $\mathbf{e}(\mathbf{c}') \stackrel{y}{\rightarrow} e_{1}$.
By Lemma~\ref{lem:adams}, we may form the translated multi-infinitesimal $y \cdot \mathbf{c}'$.
By construction, $\mathbf{e}(y \cdot \mathbf{c}') = e_{1}$.
We may therefore form the product $\mathbf{a} \times y \cdot \mathbf{c}'$.
We have that $T_{e} \subseteq \mathsf{inv}(\mathbf{a} \times y \cdot \mathbf{c}')$
and $\mathbf{e}(\mathbf{a} \times y \cdot \mathbf{c}') = e$.
The point is that we have embedded $T_{e}$ in a larger finite simple Boolean inverse monoid
which is still a subalgebra of $eSe$.

Every idempotent in $\mathbf{a} \times y \cdot \mathbf{c}'$ is $\mathscr{D}$-related to every
idempotent in $y \cdot \mathbf{c}'$.
Similarly, every idempotent in $\mathbf{b} \times x \cdot \mathbf{c}$ is $\mathscr{D}$-related to every
idempotent in $x \cdot \mathbf{c}$.
But every idempotent in $y \cdot \mathbf{c}'$ is $\mathscr{D}$-related to every
idempotent in $\mathbf{c}'$.
Similarly, every idempotent in $x \cdot \mathbf{c}$
is $\mathscr{D}$-related to every
idempotent in $\mathbf{c}$.
But $\mathbf{c}'$ is a truncation of $\mathbf{c}$.
It follows that 
every idempotent in $\mathbf{a} \times y \cdot \mathbf{c}'$ is $\mathscr{D}$-related to every
idempotent in $\mathbf{b} \times x \cdot \mathbf{c}$. 
We are therefore back to (Case 1).
\end{proof}
 
We can now prove our main theorem.

\begin{theorem}\label{thm:UHF}
let $S$ be a countably infinite Boolean inverse monoid.
Then $S$ is a UHF monoid if and only if $S$ is an AF Foulis monoid in which $\mathsf{L}(S)$ is a rational MV-algebra.
\end{theorem}
\begin{proof} 
We proved the easy direction in Proposition~\ref{prop:UHF-easy}.
It remains to prove the hard direction.
The MV-algebra $\mathsf{L}(S)$ is simple by 
Proposition~\ref{prop:locally-finite-MV}.
Thus the order on $\mathsf{L}(S)$ is linear by  Proposition~\ref{prop:simple-MV}.
By Lemma~\ref{lem:order-order}, $S/\mathscr{J}$ is linearly ordered.
By Proposition~\ref{prop:stone}, $S$ satisfies the idempotent condition.
We can write $S = \bigcup_{i=1}^{\infty} S_{i}$
where $S_{1} \subseteq S_{2} \subseteq \ldots$
and each $S_{i}$ is a matricial subalgebra.
By Lemma~\ref{lem:alex}, and repeated application of Proposition~\ref{prop:marble},
there is a finite simple subalgebra $T_{1}'$ containing $T_{1}$.
Since $T_{1}'$ is finite, there is an $i$ such that
$T_{1}' \subseteq T_{i}$.
We may likewise construct a finite simple subalgebra $T_{2}'$ containing $T_{i}$.
Continuing in this way, we have shown that $S$ is a UHF monoid.
\end{proof}

We could equally well characterize UHF monoids as follows:
they are the countably infinite simple basic Boolean inverse monoids in which  the group of units $\mathsf{U}(S)$ is locally finite and in which $\mathsf{L}(S)$ is a locally finite MV-algebra.

\section{Concluding remarks}

The authors are grateful to Dr Ganna Kudryavtseva for her suggestion to include the following discussion.
Observe first that the groups studied in \cite{KS} are precisely the groups of units of our UHF monoids;
this follows from \cite[Lemma 3.11]{LS} and the fact that the `diagonal embeddings' of \cite{KS} are restrictions to the 
groups of units of standard maps \cite{LS}.
We now apply non-commutative Stone duality \cite{Law3} (and subsequent papers).
By Proposition~\ref{prop:principal},
if $S$ is a Boolean inverse meet-monoid.
then $S$ is basic precisely when its associated Boolean groupoid is principal.
UHF monoids are simple: that is, they are $0$-simplifying and fundamental (since they are AF and so basic);
this follows by Proposition~\ref{prop:basic-properties}.
These monoids are also meet-monoids by Proposition~\ref{prop:basic-properties}.
Recall that we call the (unique) countable atomless Boolean algebra the Tarski algebra.
A countably infinite Boolean inverse meet-monoid which has a Tarski algebra of idempotents is called a Tarski monoid. 
By what we term the Dichotomy Theorem \cite[Proposition 4.4]{Lawson2016},
a simple countable Boolean inverse meet-monoid is either finite, and so isomorphic to 
a finite symmetric inverse monoid, or a Tarski monoid.
It follows that a UHF monoid is a Tarski monoid.
By \cite[Theorem 2.10]{Lawson2017} 
--- which is an interpretation of work by Matui \cite[Theorem 3.9]{Matui13} using non-commutative Stone duality ---
two  UHF monoids are isomorphic if and only if their groups of units are isomorphic.
However, as we explained above, these are just the groups studied in \cite{KS}.
This shows that, as we would expect, there is a close connection between our work on UHF monoids
and the groups defined in \cite{KS}.
However, we do not know of an abstract characterization of the groups in \cite{KS}.


\end{document}